\documentclass[12pt]{article}
\usepackage{graphicx,amssymb}
\usepackage[bookmarks=true]{hyperref}
\newcommand{\dsp}{\displaystyle}
\newcommand{\bd}{\begin{displaymath}}
\newcommand{\be}{\begin{equation}}
\newcommand{\beq}{\begin{eqnarray}}
\newcommand{\ba}{\begin{array}}
\newcommand{\ed}{\end{displaymath}}
\newcommand{\ee}{\end{equation}}
\newcommand{\eeq}{\end{eqnarray}}
\newcommand{\ea}{\end{array}}
\newcommand{\espace}{\mbox{ }}

\newcommand{\eps}{\varepsilon}
\newcommand{\abs}[1]{\left|#1\right|}
\newcommand{\Prob}{{\rm I\hspace{-0.8mm}P}}
\newcommand{\Exp}{{\rm I\hspace{-0.8mm}E}}
\newcommand{\N}{{\mathbb N}}

\newcommand{\Z}{{\mathbb Z}}
\newcommand{\R}{{\mathbb R}}

\newcommand{\dt}{\partial_t}
\newcommand{\dx}{\partial_x}
\newcommand{\eqref}[1]{(\ref{#1})}

\newtheorem{theorem}{Theorem}[section]

\newtheorem{proposition}{Proposition}[section]

\newtheorem{lemma}{Lemma}[section]
\newtheorem{corollary}{Corollary}[section]

\newtheorem{remark}{Remark}[section]
\newenvironment{proof}[2]{\espace\\{\em Proof of #1 \ref{#2}.}}{\hfill\mbox{$\square$}}
\begin{document}
\title{Euler hydrodynamics  for attractive particle
systems in random environment}
\author{C. Bahadoran$^{a,e}$, H. Guiol$^{b,e}$, K. Ravishankar$^{c,e,f}$, E. Saada$^{d,e}$}
\date{}
\maketitle
$$ \ba{l}
^a\,\mbox{\small Laboratoire de Math\'ematiques, Universit\'e Clermont 2, 63177 Aubi\`ere, France} \\
\quad \mbox{\small e-mail:
bahadora@math.univ-bpclermont.fr}\\
^b\, \mbox{\small UJF-Grenoble 1 / CNRS / Grenoble INP / TIMC-IMAG UMR 5525, Grenoble, F-38041, France
} \\
\quad \mbox{\small e-mail:
herve.guiol@imag.fr}\\
^c\, \mbox{\small Dep. of Mathematics, SUNY,
College at New Paltz, NY, 12561, USA} \\
\quad \mbox{\small e-mail:
 ravishak@newpaltz.edu}\\
^d\, \mbox{\small CNRS, UMR 8145, MAP5,
Universit\'e Paris Descartes,
Sorbonne Paris Cit\'e, France}\\
\quad \mbox{\small e-mail:
Ellen.Saada@mi.parisdescartes.fr}\\ \\
^e\, \mbox{\small  Supported by  grants ANR-07-BLAN-0230, ANR-2010-BLAN-0108, PICS 5470 }\\
^f\, \mbox{\small Supported by NSF grant DMS 0104278}\\
\ea
$$
\begin{abstract}
\noindent  We prove  quenched hydrodynamic limit under hyperbolic
time scaling for bounded attractive particle systems on $\Z$ in
random ergodic environment. Our result is a strong law of large
numbers,  that we  illustrate with  various examples. 
\end{abstract}
%
%
%
\section{Introduction}
Hydrodynamic limit  
 describes the time evolution (usually governed by a limiting PDE,
 called the hydrodynamic equation) of empirical density
fields in interacting particle systems (\emph{IPS}). For usual
models, such as the simple exclusion process, the limiting PDE is a
nonlinear diffusion equation or  hyperbolic conservation law 
(see \cite{kl} and references therein). In this context, a random environment leads to
homogeneization-like effects, where an effective diffusion matrix or
flux function is expected to capture the effect of inhomogeneity. 
Hydrodynamic limit in random environment has been widely addressed and robust methods  have been  developed in the diffusive case 
(\cite{fa,fam,fri,gj,jar,kou,nag,qua}).\\ \\  
 In the hyperbolic setting, 
due to non-existence of strong solutions and non-uniqueness of weak solutions, the key issue 
is to establish convergence to the so-called entropy solution (see e.g. \cite{serre}) of the Cauchy problem. 
The first such result without restrictive assumptions is due to \cite{fraydoun} for spatially homogeneous attractive systems
with product invariant measures.
In random environment,  the few available results 
depend on particular features of the investigated models. 
In \cite{bfl}, the authors consider the
asymmetric zero-range process with site disorder on $\Z^d$,   extending a model  introduced in \cite{ev}.
They prove a quenched hydrodynamic limit given by a hyperbolic
conservation law with an effective homogeneized flux function. To
this end, they use in particular the existence of explicit product
invariant measures for the disordered zero-range process  below
some critical value of the disorder parameter.   In \cite{ks}, extension to the supercritical case
is carried out in the totally asymmetric case with constant jump rate. 
In \cite{timo},  under a strong
mixing assumption,  the author establishes a  quenched hydrodynamic
limit for  the totally asymmetric nearest-neighbor $K$-exclusion
process on $\Z$ with  site disorder,  for which explicit
invariant measures are not known. The  last  two results rely on
  a microscopic version of the Lax-Hopf formula.  However,
the simple exclusion process beyond the totally asymmetric nearest-neighbor case, or more complex
models with state-dependent jump rates, remain outside the scope of the above approaches.  \\ \\
In this paper, we prove quenched hydrodynamics for attractive
particle systems in random environment on $\Z$ with a bounded number
of particles per site. Our method is quite robust with respect to
the model and disorder.  
We  only require the
environment to be ergodic. Besides, we are not restricted to site or
bond disorder.  However, for simplicity we treat in detail the
misanthropes' process with site disorder, and explain in the last
section how our method applies to  various other  models.
An essential difficulty for the disordered system is the
simultaneous loss of translation invariance {\em and}  lack of
knowledge  
of explicit invariant measures. Note that
even if the system without disorder has explicit invariant measures,
the disordered system in general does not, with the above exception
of the zero-range process.  In particular, one does not have an
effective characterization theorem for invariant measures of the
quenched process.
Our strategy is to prove hydrodynamic limit for a joint
disorder-particle process which, unlike the quenched process, is translation invariant.
The idea is that hydrodynamic limit for the joint process should imply quenched hydrodynamic limit.
This  is false for limits  in the usual (weak) sense,
but becomes true if a \textit{strong}  hydrodynamic limit is proved for the joint process. We are able to do it 
 by  characterizing the extremal invariant and translation invariant measures of the joint process, 
and by  adapting the tools developed in \cite{bgrs3}.\\ \\
The paper is organized as follows.  
In Section \ref{sec_results}, we define the model and state our main
result. Section \ref{sec_disorder_particle} is devoted to the study
of the joint disorder-particle process  and characterization of its
invariant measures. The hydrodynamic limit is proved in Section
\ref{proof_hydro}. Finally, in Section \ref{sec:general} we consider
 models other than  the misanthropes' process:   We detail
generalizations of misanthropes and $k$-step exclusion processes, as
well as a traffic model.
\section{Notation and results}\label{sec_results}
Throughout this paper $\N=\{1,2,...\}$ will denote the set of
natural numbers, and $\Z^+=\{0,1,2,...\}$ the set of non-negative
integers.
The integer part  $\lfloor x\rfloor\in\Z$ of $x\in\R$    is
uniquely defined by $\lfloor x\rfloor \leq x<\lfloor x\rfloor +1$.
We consider particle  configurations on $\Z$ with at most $K$
particles per site, $K\in\N$. Thus the state space of the process is
${\mathbf X}=\{0,1,\cdots,K\}^{\Z}$, which we endow with the product
topology, that makes ${\mathbf X}$ a compact metrisable space, 
with the product (partial) order.\\ \\
 The set  $\mathbf A$ of environments 
 is a compact metric space endowed with its Borel
$\sigma$-field.
 A function $f$ defined on ${\mathbf A}\times{\mathbf X}$
 (resp. $g$   on ${\mathbf A}\times{\mathbf X}^2$)
 is called {\em local} if there is a finite subset
$\Lambda$ of $\Z$ such that $f(\alpha,\eta)$  depends only on
$\alpha$ and $(\eta(x),x\in\Lambda)$
 (resp. $g(\alpha,\eta,\xi)$ depends only on
$\alpha$ and $(\eta(x),\xi(x),x\in\Lambda)$).
We denote by $\tau_x$ either the spatial translation operator on the
real line for $x\in\R$, defined by $\tau_x y=x+y$, or its
restriction to  $x\in\Z$. By extension, if $f$ is a function defined
on $\Z$ (resp. $\R$), we set $\tau_x f=f\circ\tau_x$ for $x\in\Z$
(resp. $\R$). In the sequel this will be applied to particle
configurations $\eta\in\mathbf X$, disorder configurations
$\alpha\in\mathbf{A}$, or joint disorder-particle configurations
$(\alpha,\eta)\in\mathbf{A}\times\mathbf{X}$. In the latter case,
unless mentioned explicitely, $\tau_x$  
 applies simultaneously
to both components.\\ \\
If $\tau_x$ acts on some set and $\mu$ is a measure on this set,
$\tau_x\mu=\mu\circ\tau_{x}^{-1}$.
We let ${\mathcal M}^+(\R)$ denote the set of  nonnegative
measures on $\R$ equipped with the metrizable topology of vague
convergence, defined by convergence on continuous test functions
with compact support. The set of probability measures on
$\mathbf{X}$ is denoted by ${\mathcal P}(\mathbf{X})$. If $\eta$ is
an ${\mathbf X}$-valued random variable and $\nu\in{\mathcal
P}(\mathbf{X})$, we write $\eta\sim\nu$ to specify that $\eta$ has
distribution $\nu$.  Similarly, for
$\alpha\in\mathbf{A}, Q\in{\mathcal P}(\mathbf{A})$, $\alpha\sim Q$
means that $\alpha$ has distribution $Q$. \\ \\
 A sequence  $(\nu_n,n\in\N)$ of probability measures on
${\mathbf X}$ converges weakly to some $\nu\in{\mathcal
P}(\mathbf{X})$, if and only if  $\lim_{n\to\infty}\int
f\,d\nu_n=\int f\,d\nu$  for every continuous function $f$ on
${\mathbf X}$. The topology of weak convergence is metrizable and
makes ${\mathcal P}({\mathbf X})$ compact.
A partial stochastic order is defined on ${\mathcal P}(\mathbf{X})$;
namely, for $\mu_1,\mu_2\in{\mathcal P}(\mathbf{X})$,
we write $\mu_1\leq\mu_2$ if the following equivalent conditions
hold (see {\em e.g.}  \cite{lig1, strassen}):  \textit{(i)} For
every non-decreasing nonnegative function $f$ on $\mathbf X$, $\int
f\,d\mu_1\leq\int f\,d\mu_2$. \textit{(ii)} There exists a coupling
measure  $\overline{\mu}$ on $
\mathbf {X}\times\mathbf {X}$ with marginals $\mu_1$ and $\mu_2$,
such that $\overline{\mu}\{(\eta,\xi):\,\eta\leq\xi\}=1$.
\\ \\
In the following model, we fix a constant $c>0$ and define ${\mathbf
A}=[c,1/c]^\Z$  to be  the set of environments (or disorders)
$\alpha=(\alpha(x):\,x\in\Z)$ such that
\be\label{alpha_bounds}\forall x\in\Z,\quad c\leq\alpha(x)\leq
c^{-1}\ee
For each realization $\alpha\in{\mathbf A}$ of the disorder, the
{\em quenched process} $(\eta_t)_{t\geq 0}$ is a Feller process on
$\mathbf X$ with generator  given by, for any local function $f$ on
${\mathbf X}$,
\be \label{generator} L_\alpha f(\eta)=\sum_{x,y\in{\Z}}\alpha(x)
p(y-x)b(\eta(x),\eta(y)) \left[ f\left(\eta^{x,y} \right)-f(\eta)
\right] \ee
 where $\eta^{x,y}$ denotes
the new state after a particle has jumped from $x$ to $y$ (that is
$\eta^{x,y}(x)=\eta(x)-1,\,\eta^{x,y}(y)=\eta(y)+1,\,
\eta^{x,y}(z)=\eta(z)$ otherwise), the particles' jump
kernel $p$ is  a probability distribution on $\Z$,  and $b\ :\
\Z^+\times\Z^+\to \R^+$ is the jump rate.
We assume that $p$ and $b$ satisfy:  \\ \\
{\em (A1)} Irreducibility: For every $z\in\Z$, 
$\sum_{n\in\N}[p^{*n}(z)+p^{*n}(-z)]>0$, 
where $*n$ denotes $n$-th convolution power;\\
{\em (A2)} finite mean: $\sum_{z\in\Z}\abs{z}p(z)<+\infty$;\\
{\em (A3)} $K$-exclusion rule: $b(0,.)=0,\,b(.,K)=0$;\\
{\em (A4)} non-degeneracy:  $b(1,K-1)>0$; \\
{\em (A5)} attractiveness: $b$ is nondecreasing (nonincreasing) in its first
(second) argument. \\ \\
For the graphical construction of the system given by
\eqref{generator}
(see \cite{bgrs3} and references therein),  let us introduce a
general framework that applies to a larger class of models (see
Section \ref{sec:general}  below).
 Given a  measurable space $(\mathcal V,{\mathcal F}_\mathcal V,m)$,
  for $m$  a finite nonnegative measure,
we consider the probability space $(\Omega,{\mathcal F},\Prob)$ of
locally finite point measures $\omega(dt,dx,dv)$ on
$\R^+\times\Z\times\mathcal V$, where $\mathcal F$ is generated by the
mappings $\omega\mapsto\omega(S)$ for Borel sets $S$  of $\R^+\times\Z\times\mathcal V$,  and $\Prob$
makes $\omega$ a Poisson process with intensity
\[
M(dt,dx,dv)=
\lambda_{\R^+}(dt)\lambda_\Z(dx)m(dv) \]
  denoting by $\lambda$  either   the Lebesgue or
 the counting measure. We write $\Exp$  for  expectation 
 with respect to $\Prob$. For the particular model \eqref{generator}
we take
\be\label{special_choice}
\mathcal V:=\Z\times[0,1],\quad
v=(z,u)\in\mathcal V,\quad
m(dv)=c^{-1}||b||_\infty p(dz)\lambda_{[0,1]}(du)
\ee
Thanks to assumption {\em (A2)}, for $\Prob$-a.e. $\omega$, there
exists a unique mapping
\be \label{unique_mapping} (\alpha,\eta_0,t)\in{{\mathbf
A}}\times{\mathbf
X}\times\R^+\mapsto\eta_t=\eta_t(\alpha,\eta_0,\omega)\in{\mathbf X}
\ee
satisfying: \textit{(a)} $t\mapsto\eta_t(\alpha,\eta_0,\omega)$ is
right-continuous; \textit{(b)} $\eta_0(\alpha,\eta_0,\omega)=\eta_0$; \textit{(c)} the
particle configuration  is updated at points  $(t,x,v)\in\omega$ (and only at such  points;
by $(t,x,v)\in\omega$ we mean  $\omega\{(t,x,v)\}=1$)  according to the rule
\be\label{update_rule}
\eta_t(\alpha,\eta_0,\omega)={\mathcal T}^{\alpha,x,v}\eta_{t^-}(\alpha,\eta_0,\omega)
\ee
where, for $v=(z,u)\in\mathcal V$,  ${\mathcal T}^{\alpha,x,v}$ is defined by 
\be\label{update_misanthrope}
{\mathcal T}^{\alpha,x,v}\eta=\left\{
\ba{lll}
\eta^{x,x+z} & \mbox{if} & u<\displaystyle\alpha(x){\frac{b(\eta(x),\eta(x+z))}
{c^{-1}||b||_\infty}}\\
\eta &  & \mbox{otherwise}
\ea
\right.
\ee
Notice the shift commutation property 
\be\label{shift_t}
{\mathcal T}^{\tau_x\alpha,y,v}\tau_x=\tau_x{\mathcal T}^{\alpha,y+x,v}
\ee
where $\tau_x$  on the right-hand side acts only on $\eta$.
By assumption {\em (A5)},
\be\label{attractive_0}{\mathcal T}^{\alpha,x,v}:{\mathbf X}\to{\mathbf X}\mbox{ is nondecreasing}\ee
Hence,
\be \label{attractive_1}
(\alpha,\eta_0,t)\mapsto\eta_t(\alpha,\eta_0,\omega) \mbox{ is
nondecreasing w.r.t. } \eta_0\ee
For every $\alpha\in{\mathbf A}$, under $\Prob$,
$(\eta_t(\alpha,\eta_0,\omega))_{t\geq 0}$ is a Feller process with
generator
\be\label{gengen}
L_\alpha f(\eta)=\sum_{x\in\Z}\int_{\mathcal V}\left[
f\left({\mathcal T}^{\alpha,x,v}\eta\right)-f(\eta)
\right]m(dv) \ee
With \eqref{update_misanthrope}, \eqref{gengen} reduces to \eqref{generator}.
Thus for any $t\in\R^+$ and  continuous function $f$ on
${\mathbf X}$,
$\Exp[f(\eta_t(\alpha,\eta_0,\omega))]=S_{\alpha}(t)f(\eta_0)$, where $S_\alpha$ denotes the semigroup generated by $L_\alpha$.
{}From \eqref{attractive_1},  for $\mu_1,\mu_2\in{\mathcal
P}(\mathbf{X})$,
\be \label{attractive_2} \mu_1\leq\mu_2\Rightarrow\forall
t\in\R^+,\,\mu_1 S_\alpha(t)\leq\mu_2 S_\alpha(t) \ee
Property \eqref{attractive_2} is
usually called {\em attractiveness}. Condition \eqref{attractive_0} implies the stronger {\em complete monotonicity} property (\cite{fm, dplm}),
that is, existence of a monotone Markov coupling for an arbitrary number of processes with generator \eqref{generator}, see  \eqref{coupling_t} below;
we also say that the process is \textit{strongly attractive}. \\ \\
Let $N\in\N$ be the scaling parameter for the hydrodynamic limit,
that is, the inverse of the macroscopic distance between two
consecutive sites. The empirical measure of a configuration $\eta$
viewed on scale $N$ is given by 
\[
\pi^N(\eta)(dx)=N^{-1}\sum_{y\in\Z}\eta(y)\delta_{y/N}(dx)
\in{\mathcal M}^+(\R) \]
 where, for $x\in\R$, $\delta_x$ denotes the Dirac measure at $x$. \\
Our main result is
\begin{theorem}\label{th:hydro}
Assume $p(.)$ has finite third moment.
Let $Q$ be an ergodic probability distribution on $\mathbf A$. Then
there exists a Lipschitz-continuous function  $G^Q$ on
$[0,K]$ defined below (depending only on $p(.)$, $b(.,.)$ and $Q$)
such that the following holds.
Let $(\eta^N_0,\,N\in\N)$ be a sequence of ${\mathbf X}$-valued random
variables on a probability space $(\Omega_0,\mathcal F_0,\Prob_0)$
such that
\be\label{initial_profile_vague}
\lim_{N\to\infty}\pi^N(\eta^N_0)(dx)=
u_0(.)dx\quad\Prob_0\mbox{-a.s.}\ee
for some measurable $[0,K]$-valued profile $u_0(.)$.
Then for  $Q$-a.e.  $\alpha\in{\mathbf A}$, the
$\Prob_0\otimes\Prob$-a.s. convergence
\[
\lim_{N\to\infty}\pi^N(\eta_{Nt}(\alpha,\eta^N_0(\omega_0),\omega))(dx)=u(.,t)dx
\]
 holds uniformly on all bounded time intervals, where $(x,t)\mapsto
u(x,t)$ denotes the unique entropy solution with initial condition
$u_0$ to the conservation law
\be \label{hydrodynamics} \dt u+\dx[G^Q(u)]=0 \ee
\end{theorem}
We refer the reader (for instance) to \cite{serre} for the
definition of entropy solutions.
To define the {\sl macroscopic flux} $G^Q$, let
 the {\sl microscopic flux}  through site 0 be
\beq\label{def_f}
j(\alpha,\eta)&=&j^+(\alpha,\eta)-j^-(\alpha,\eta)\\\nonumber
j^+(\alpha,\eta) &=& \sum_{y,z\in\Z:\,y\leq 0<y+z}\alpha(y)
p(z)b(\eta(y),\eta(y+z))\\\nonumber
 j^-(\alpha,\eta) & = &
\sum_{y,z\in\Z:\,y+z\leq 0<y}\alpha(y) p(z)b(\eta(y),\eta(y+z)) \eeq
We  will show in Corollary \ref{corollary_invariant} below
that there exists a closed subset  $\mathcal R^Q$ of $[0,K]$, a
subset $\mathbf{\widetilde{A}}^Q$ of ${\mathbf A}$ with
$Q$-probability $1$ (both depending also on $p(.)$ and $b(.,.)$),
and a family of probability measures
$(\nu^{Q,\rho}_\alpha:\,\alpha\in\mathbf{\widetilde{A}}^Q,\rho\in\mathcal
R^Q)$ on ${\mathbf X}$,  such
that, for every $\rho\in\mathcal R^Q$:\\
\label{properties_b}
\noindent {\em (B1)} For every  $\alpha\in\mathbf{\widetilde{A}}^Q$, $\nu^{Q,\rho}_\alpha$
is an
invariant measure for $L_\alpha$.\\
{\em (B2)} For every  $\alpha\in\mathbf{\widetilde{A}}^Q$, $\nu^{Q,\rho}_\alpha$-a.s.,
$$\lim_{l\to\infty}(2l+1)^{-1}\sum_{x\in\Z:\,|x|\leq
l}\eta(x)=\rho$$
{\em (B3)} The quantity
\be \label{flux} G^Q_\alpha(\rho):=\int j(\alpha,\eta)\nu^{Q,\rho}_\alpha(d\eta) \ee
does not depend on  $\alpha\in\mathbf{\widetilde{A}}^Q$.
Hence we define $G^Q(\rho)$ as \eqref{flux} for $\rho\in\mathcal R^Q$ and
extend it by linear interpolation on the complement of $\mathcal R^Q$,
which is a finite or countably infinite union of disjoint open
intervals. \\ \\
The function $G^Q$ is Lipschitz continuous  (see Remark \ref{remark_lipschitz} below), which is  the minimum regularity required
for the classical theory of entropy solutions.  We cannot say more about $G^Q$ in general, because
the measures $\nu^{Q,\rho}_\alpha$ are not explicit. This is true even in the spatially homogeneous case $\alpha(x)\equiv 1$, unless $b$ satisfies additional algebraic relations introduced in \cite{coc}. In the absence of disorder, for the exclusion process and a few simple models satisfying these relations (see for instance \cite[Section 4]{bgrs1}), we have an explicit flux function. Nevertheless, invariant measures are no longer computable when introducing disorder, so that the effect of the latter on the flux function is difficult to evaluate.
However, in the special case $b(n,m)={\bf 1}_{\{n>0\}}{\bf 1}_{\{m<K\}},\,p(1)=1$, $G^Q$ is shown to be concave in \cite{timo}, as a consequence of the variational approach used  to derive hydrodynamic limit. But this approach does not apply to the models we
consider in the present paper. 
\section{The disorder-particle process}
\label{sec_disorder_particle}
In this section we study invariant measures for the 
markovian  {\em joint
process} $(\alpha_t,\eta_t)_{t\ge 0}$ on ${\mathbf A}\times {\mathbf
X}$ with generator given by, for any local function $f$ on ${\mathbf
A}\times{\mathbf X}$,
\be \label{generator_joint-general} L
f(\alpha,\eta)=\sum_{x\in{\Z}}\int_{\mathcal V}\left[
f\left(\alpha,{\mathcal T}^{\alpha,x,v}\eta\right)-f(\alpha,\eta)
\right]m(dv) \ee
that is, for the particular model \eqref{update_misanthrope},
\be \label{generator_joint} L
f(\alpha,\eta)=\sum_{x,y\in{\Z}}\alpha(x) p(y-x)b(\eta(x),\eta(y))
\left[ f\left(\alpha,\eta^{x,y} \right)-f(\alpha,\eta) \right] \ee
We denote by
$(S(t),t\in\R^+)$ the semigroup generated by $L$. Given
$\alpha_0=\alpha$, this dynamics simply means that $\alpha_t=\alpha$
for all $t\geq 0$, while $(\eta_t)_{t\ge 0}$  is a Markov process with
generator $L_\alpha$ given by \eqref{generator}. Note that $L$ is
\textit{translation invariant}, that is 
\be\label{eq:L-transl-inv}\tau_x L=L\tau_x\ee
 where $\tau_x$
acts jointly on $(\alpha,\eta)$. This is equivalent to a  \textit{shift commutation
property} for the quenched dynamics:
\be\label{commutation}
L_\alpha\tau_x=\tau_x L_{\tau_x\alpha}
\ee
where, since $L_\alpha$ is a Markov generator on ${\mathbf X}$,  the first $\tau_x$ on the r.h.s.  acts only on $\eta$.
We need to introduce a conditional stochastic order.  For the sequel,
we define the set $\overline{\mathcal O}  =  \overline{\mathcal O}_+\cup\overline{\mathcal O}_-$, where
\begin{eqnarray}
\overline{\mathcal O}_+ & = & \{(\alpha,\eta,\xi)\in{\mathbf
A}\times{\mathbf X}^2: \eta\leq\xi \}\cr \overline{\mathcal O}_- & =
& \{(\alpha,\eta,\xi)\in{\mathbf A}\times{\mathbf X}^2: \xi\leq\eta
\} \label{def_ordered_set}
\end{eqnarray}
\begin{lemma}
\label{lemma_conditional}
For two probability measures $\mu^1=\mu^1(d\alpha,d\eta)$,
$\mu^2=\mu^2(d\alpha,d\eta)$ on ${\mathbf A}\times{\mathbf X}$, the
following properties  (denoted  by $\mu^1\ll\mu^2$) are
equivalent:
(i) For every bounded measurable local function $f$ on ${\mathbf
A}\times{\mathbf X}$, such that $f(\alpha,.)$ is nondecreasing for
all $\alpha\in{\mathbf A}$, we have $\int f\,d\mu^1\leq\int
f\,d\mu^2$.
(ii) The measures $\mu^1$ and $\mu^2$ have a common
$\alpha$-marginal  $Q$, and
$\mu^1(d\eta|\alpha)\leq\mu^2(d\eta|\alpha)$ for $Q$-a.e.
$\alpha\in{\mathbf A}$.
(iii) There exists a coupling measure $\overline{\mu}(d\alpha,d\eta,d\xi)$ supported on
$\overline{\mathcal O}_+$ under which $(\alpha,\eta)\sim\mu^1$ and $(\alpha,\xi)\sim\mu^2$.
\end{lemma}
\begin{proof}{lemma}{lemma_conditional}
\textit{(ii)}$\Rightarrow$\textit{(i)} follows from conditioning.
For \textit{(i)}$\Rightarrow$\textit{(ii)}, consider
$f(\alpha,\eta)=g(\alpha)h(\eta)$, where $g$ is a nonnegative
measurable function on $\mathbf A$ and $h$ a nondecreasing local
function on ${\mathbf X}$. Specializing to $h\equiv 1$, using both
$f$ and $-f$ in \textit{(i)}, we obtain
$$\dsp\int g(\alpha)\mu^1(d\alpha,d\eta)= \int
g(\alpha)\mu^2(d\alpha,d\eta)$$
Thus $\mu^1$ and $\mu^2$ have a common $\alpha$-marginal  $Q$. Now with a
general $h$, by conditioning, we have
$$
\int g(\alpha)\left(\int
h(\eta)\mu^1(d\eta|\alpha)\right)Q(d\alpha)\leq\int
g(\alpha)\left(\int h(\eta)\mu^2(d\eta|\alpha)\right)Q(d\alpha)
$$
Thus, for any nondecreasing local function $h$ on ${\mathbf X}$,
$$\int
h(\eta)\mu^1(d\eta|\alpha)\leq\int h(\eta)\mu^2(d\eta|\alpha)$$
holds  $Q(d\alpha)$-a.e. Since the set of nondecreasing local
functions on ${\mathbf X}$ has a countable dense subset (w.r.t.
uniform convergence), we can exchange ``for any $h$'' and
``$Q$-a.e.'' In other words,
$\mu^1(d\eta|\alpha)\leq\mu^2(d\eta|\alpha)$ for  $Q$-a.e.
$\alpha\in\mathbf A$.\\ \\
For \textit{(ii)}$\Rightarrow$\textit{(iii)}, by Strassen's theorem
(\cite{strassen}), for $Q$-a.e. $\alpha\in{\mathbf A}$, there exists
a coupling measure $\overline{\mu}_\alpha(d\eta,d\xi)$  on ${\mathbf
X}^2$  under which $\eta\sim\mu^1(.|\alpha)$,
$\xi\sim\mu^2(.|\alpha)$, and $\eta\leq\xi$ a.s. Then
$\overline{\mu}(d\alpha,d\eta,d\xi):=\int_{\mathbf
A}[\delta_\beta(d\alpha)\overline{\mu}_\alpha(d\eta,d\xi)]Q(d\beta)$
yields the desired coupling.
\textit{(iii)}$\Rightarrow$\textit{(i)} is straightforward.
\end{proof}
\mbox{}\\ \\
We now state the main result of this section.  Let $\mathcal I_L$,
$\mathcal S$ and ${\mathcal S}^{\mathbf A}$ denote the sets of
probability measures  that are respectively  invariant for $L$,
shift-invariant
 on $\mathbf{A}\times\mathbf{X}$ and
shift-invariant  on $\mathbf A$.
\begin{proposition}\label{invariant}
For every $Q\in{\mathcal S}^{\mathbf A}_e$, there exists a closed
subset $\mathcal R^Q$ of $[0,K]$ containing $0$ and $K$, such that
\begin{eqnarray*}
\left(\mathcal I_{L}\cap\mathcal S\right)_e & = &
\left\{\nu^{Q,\rho},\,Q\in{\mathcal S}^{\mathbf
A}_e,\,\rho\in{\mathcal R}^Q\right\}
\end{eqnarray*}
where index $e$ denotes the set of extremal elements, and
$(\nu^{Q,\rho}:\,\rho\in\mathcal R^Q)$ is a family of
shift-invariant measures on ${\mathbf A}\times{\mathbf X}$, weakly
continuous with respect to $\rho$, such that
\beq\label{densite-rho}
\int\eta(0)\nu^{Q,\rho}(d\alpha,d\eta)&=&\rho\\
\label{Rezakhanlou}
\lim_{l\to\infty}(2l+1)^{-1}\sum_{x\in\Z:|x|\le l}\eta(x)&=&\rho,
\quad \nu^{Q,\rho}-\mbox{a.s.} \\
\label{ordered_measures}\rho\leq\rho'&\Rightarrow&\nu^{Q,\rho}\ll\nu^{Q,\rho'}\eeq
\end{proposition}
For $\rho=0\in {\mathcal R^Q}$  (resp. $\rho=K\in {\mathcal R^Q}$)
we get the invariant distribution $\delta_0^{\otimes\Z}$ (resp.
$\delta_K^{\otimes\Z}$), the deterministic distribution of
the configuration with no particles (resp. with maximum number
of particles $K$ everywhere).
\begin{remark}
The set ${\mathcal R^Q}$ and measures $\nu^{Q,\rho}$ also depend on
$p(.)$ and $b(.,.)$, but we did not reflect this in the notation
because only $Q$ varies in Proposition \ref{invariant}.
\end{remark}
\begin{corollary}
\label{corollary_invariant}
(i)  The family of probability measures
$\nu^{Q,\rho}_\alpha(.):=\nu^{Q,\rho}(.|\alpha)$  on ${\mathbf X}$
satisfies properties
(B1)--(B3) on page \pageref{properties_b};
(ii)  for $\rho\in\mathcal R^Q$,  $G^Q(\rho)=\int
j(\alpha,\eta)\nu^{Q,\rho}(d\alpha,d\eta)$.
\end{corollary}
\begin{remark}
\label{remark_flux}
By \textit{(ii)} of Corollary \ref{invariant}, and shift-invariance of $\nu^{Q,\rho}(d\alpha,d\eta)$, 
\be
G^Q(\rho) = \int j(\alpha,\eta)\nu^{Q,\rho}(d\alpha,d\eta)
=\int \widetilde{\jmath}(\alpha,\eta)\nu^{Q,\rho}(d\alpha,d\eta)
\ee
for every $\rho\in\mathcal R^Q$, where
\be\label{other_flux}
\widetilde{\jmath}(\alpha,\eta):=\alpha(0)\sum_{z\in\Z}zp(z)b(\eta(0),\eta(z))
\ee
Thus one can alternatively take $\widetilde{\jmath}(\alpha,\eta)$ as a microscopic flux function
 (we refer to \cite[p. 1347]{bgrs2} for an analogous remark in the non-disordered
setting).
\end{remark}
\begin{proof}{corollary}{corollary_invariant}
Properties {\em (B1)} and {\em (B2)} follow from Proposition
\ref{invariant} by conditioning  (here and after,  we proceed as in
the proof of Lemma \ref{lemma_conditional}).  By translation
invariance of  $\nu^{Q,\rho}$  and conditioning we have, for
$Q$-a.e. $\alpha\in\mathbf A$,
\be\label{commutation_inv}
\tau_x\nu^{Q,\rho}_\alpha=\nu^{Q,\rho}_{\tau_x\alpha} \ee
where $\tau_x$ on the l.h.s. acts on ${\mathbf X}$. For property
{\em (B3)}
the result will follow from ergodicity of $Q$ once we
show that, for every $\rho\in\mathcal R^Q$,
$G^Q_\alpha(\rho)=G^Q_{\tau_1\alpha}(\rho)$ holds $Q$-a.s. To this end
we note that, as a result of \eqref{commutation},
$$L_\alpha[\eta(1)]=j(\alpha,\eta)-j(\tau_1\alpha,\tau_1\eta)$$
Taking expectation w.r.t. invariant measure $\nu_\alpha^{Q,\rho}$, and
using \eqref{commutation_inv}, we obtain
\begin{eqnarray*}
G^Q_\alpha(\rho)=\int j(\alpha,\eta)\nu^{Q,\rho}_\alpha(d\eta) & = & \int
j(\tau_1\alpha,\tau_1\eta)\nu^{Q,\rho}_\alpha(d\eta)\\
& = & \int
j(\tau_1\alpha,\eta)\nu^{Q,\rho}_{\tau_1\alpha}(d\eta)=G^Q_{\tau_1\alpha}(\rho)
\end{eqnarray*}
\end{proof}
\mbox{}\\ \\
To prove Proposition \ref{invariant}, we need some definitions and lemmas.
For every $\alpha\in{\mathbf A}$, we denote by $\overline{L}_\alpha$
the coupled generator on ${\mathbf X}^2$  given by
\be\label{coupling_t}
\overline{L}_\alpha f(\eta,\xi):=\sum_{x\in\Z}\int_{\mathcal V}\left[
f\left({\mathcal T}^{\alpha,x,v}\eta,{\mathcal T}^{\alpha,x,v}\xi\right)-f(\eta,\xi)
\right]m(dv)
\ee
for any local function $f$ on ${\mathbf X}^2$.
For the particular model \eqref{update_misanthrope}, this is equivalent to the ``basic coupling'' of $L_\alpha$
 defined in \cite{coc}, namely  $\overline{L}_\alpha=\sum_{x,y\in\Z:\,x\neq y}\overline{L}^{x,y}_\alpha$, with
$\overline{L}^{x,y}_\alpha f(\eta,\xi)$ given by
\begin{eqnarray}\label{xy-gen-coupl}
&  & \alpha(x)p(y-x)[b(\eta(x),\eta(y))\wedge b(\xi(x),\xi(y))]\left[f(\eta^{x,y},\xi^{x,y})-f(\eta,\xi)\right]\cr
& + & \alpha(x)p(y-x)[b(\eta(x),\eta(y))-b(\xi(x),\xi(y))]^+\left[f(\eta^{x,y},\xi)-f(\eta,\xi)\right]\cr
& + &  \alpha(x)p(y-x)[b(\xi(x),\xi(y))- b(\eta(x),\eta(y))]^+\left[f(\eta,\xi^{x,y})-f(\eta,\xi)\right]
\end{eqnarray}
 If $(\eta_t,\xi_t)$ is a Markov process with generator $\overline{L}_\alpha$,  and $\eta_0\leq\xi_0$, then
$\eta_t\leq\xi_t$ a.s. for every $t>0$. We indicate this by saying
 that $\overline{L}_\alpha$ is a \textit{monotone coupling} of
$L_\alpha$.
We denote by $\overline{L}$ the coupled generator for the joint
process $(\alpha_t,\eta_t,\xi_t)_{t\ge 0}$ on ${\mathbf
A}\times{\mathbf X}^2$ defined by
\be\label{joint_coupling}\overline{L}f(\alpha,\eta,\xi)=(\overline{L}_\alpha
f(\alpha,.))(\eta,\xi)\ee
for any local function $f$ on ${\mathbf A}\times{\mathbf X}^2$.
Given $\alpha_0=\alpha$, this means that $\alpha_t=\alpha$ for all
$t\geq 0$, while $(\eta_t,\xi_t)_{t\ge 0}$ is a Markov process with
generator $\overline{L}_\alpha$.  Let  $\overline{S}(t)$ denote the
semigroup generated by $\overline{L}$.
 We denote by  $\overline{\mathcal S}$ the set of probability measures on
${\mathbf A}\times{\mathbf X}^2$ that are invariant by space shift
$\tau_x(\alpha,\eta,\xi)=(\tau_x\alpha,\tau_x\eta,\tau_x\xi)$.
In the following, if $\overline{\nu}(d\alpha,d\eta,d\xi)$ is a
probability measure on ${\mathbf A}\times{\mathbf X}^2$,
$\overline{\nu}_1$, $\overline{\nu}_2$ and  $\overline{\nu}_3$
(resp. $\overline{\nu}_{12}$ and $\overline{\nu}_{13}$)   denote
marginal distributions of  $\alpha$, $\eta$ and $\xi$ (resp. $(\alpha,\eta)$
and $(\alpha,\xi)$) under $\overline{\nu}$.
\begin{lemma}
\label{coupling_extremal}
Let $\mu',\mu''\in(\mathcal I_{L}\cap\mathcal S)_e$ with a
common $\alpha$-marginal $Q$. Then there exists
$\overline{\nu}\in\left(\mathcal
I_{\overline{L}}\cap\overline{\mathcal S}\right)_e$
such that $\overline{\nu}_{12}=\mu'$ and
$\overline{\nu}_{13}=\mu''$.
\end{lemma}
\begin{proof}{lemma}{coupling_extremal}
Let $\overline{\mathcal M}(\mu',\mu'')$ denote the set of
probability measures $\overline{\nu}\in\mathcal
I_{\overline{L}}\cap\overline{\mathcal S}$ with
$\overline{\nu}_{12}=\mu'$ and $\overline{\nu}_{13}=\mu''$. We show
that $\overline{\mathcal M}(\mu',\mu'')$ is a nonempty set.
Set
$\overline{\nu}^0(d\alpha,d\eta,d\xi):=Q(d\alpha)[\mu'(d\eta|\alpha)\otimes\mu''(d\xi|\alpha)]$.
Then
$\overline{\nu}^0_{12}=\mu'$, $\overline{\nu}^0_{13}=\mu''$ and
$\overline{\nu}^0\in\overline{\mathcal S}$. Let
$$
\overline{\nu}^t:=\frac{1}{t}\int_0^t\overline{\nu}^0\overline{S}(s)
ds
$$
The set $\{\overline{\nu}^t,t>0\}$ is relatively compact because
$\overline{\nu}^t_1=Q$ is independent of $t$ and, for $i\in\{2,3\}$,
$\overline{\nu}^t_i\leq\delta_K^{\otimes\Z}$.
Let $\overline{\nu}^\infty$ be any subsequential weak limit of
$\overline{\nu}^t$ as $t\to\infty$. Then $\overline{\nu}^\infty$
retains the above properties of $\overline{\nu}^0$, and 
 $\overline{\nu}^\infty\in \mathcal I_{\overline{L}}$,
thus $\overline{\nu}^\infty\in\overline{\mathcal M}(\mu',\mu'')$.
Let $\overline{\nu}$ be an extremal element of  the compact convex
set $\overline{\mathcal M}(\mu',\mu'')$.  We now prove that
$\overline{\nu}\in\left(\mathcal
I_{\overline{L}}\cap\overline{\mathcal S}\right)_e$.
Assume  there exist $\lambda\in (0,1)$ and probability measures
$\overline{\nu}^l$, $\overline{\nu}^r$ on ${\bf A}\times{\bf X}^2$, such that
\be\label{decomp_2}
\overline{\nu}=\lambda \overline{\nu}^l+(1-\lambda)\overline{\nu}^r
\ee
with $\overline{\nu}^i\in\mathcal
I_{\overline{L}}\cap\overline{\mathcal S}$ for $i\in\{l,r\}$.
Since $\overline{\nu}\in\overline{\mathcal M}(\mu',\mu'')$, the 
projections of \eqref{decomp_2} on  $(\alpha,\eta)$ and $(\alpha,\xi)$ yield
\begin{eqnarray}
\mu' & = & \lambda \overline{\nu}^l_{12}+(1-\lambda)\overline{\nu}^r_{12}\label{decomp_21}\\
\mu'' & = & \lambda
\overline{\nu}^l_{13}+(1-\lambda)\overline{\nu}^r_{13}\label{decomp_22}
\end{eqnarray}
 For
$i\in\{l,r\}$, $\overline{\nu}^i\in\mathcal
I_{\overline{L}}\cap\overline{\mathcal S}$ implies
$\overline{\nu}^i_{1j}\in\mathcal I_{L}\cap\mathcal S$ for
$j\in\{2,3\}$. Since $\mu',\mu''$ belong to $(\mathcal
I_{L}\cap\mathcal S)_e$,  
$\overline{\nu}^i_{12}=\mu'$, $\overline{\nu}^i_{13}=\mu''$ by
\eqref{decomp_21}--\eqref{decomp_22}, that is,
$\overline{\nu}^i\in\overline{\mathcal M}(\mu',\mu'')$. Since
$\overline{\nu}$ is extremal in $\overline{\mathcal M}(\mu',\mu'')$,
\eqref{decomp_2} yields
$\overline{\nu}^l=\overline{\nu}^r=\overline{\nu}$.
\end{proof}
 \begin{lemma}\label{stable}
Let $\nu$ be a stationary distribution for some Markov transition
semigroup, and $(X_t)_{t\ge 0}$ be a Markov process associated to this
semigroup with initial distribution $\nu$.
Assume $A$ is a subset of $E$ such that,  
for every $t>0$, ${\bf 1}_A(X_t)\geq {\bf 1}_A(X_0)$ almost surely. Then
$\nu_A(dx)=\nu(dx|x\in A)$ and $\nu_{A^c}(dx)=\nu(dx|x\in A^c)$ are stationary
for the considered  semigroup.
\end{lemma}
\begin{proof}{lemma}{stable}
Since $(X_t)_{t\ge 0}$ is stationary, we have $\Exp [{\bf
1}_A(X_t)]=\Exp [{\bf 1}_A(X_0)]$,  thus ${\bf 1}_A(X_0)={\bf
1}_A(X_t)$ and ${\bf 1}_{A^c}(X_t)={\bf 1}_{A^c}(X_0)$  almost
surely. Stationarity of $\nu_A$ amounts to $\dsp \Exp[f(X_t)|X_0\in
A]=\Exp[f(X_0)|X_0\in A]$ for every bounded $f$. We conclude with
$$
\begin{array}{lllll}
\Exp[f(X_t)|X_0\in A]
& = & \displaystyle\frac{\Exp[f(X_t){\bf 1}_A(X_0)]}{\Exp[{\bf 1}_A(X_0)]}
& = & \displaystyle\frac{\Exp[f(X_t){\bf 1}_A(X_t)]}{\Exp[{\bf 1}_A(X_0)]}\\
 & = & \displaystyle\frac{\Exp[f(X_0){\bf 1}_A(X_0)]}{\Exp[{\bf 1}_A(X_0)]}
& = & \displaystyle\Exp[f(X_0)|X_0\in A]
\end{array}
$$
\end{proof}
\begin{lemma}
\label{order_extremal}
Let $\overline{\nu}\in\left(\mathcal
I_{\overline{L}}\cap\overline{\mathcal S}\right)_e$. Then
$\overline{\nu}\left(\overline{\mathcal O}_+\right)$ and
$\overline{\nu}\left(\overline{\mathcal O}_-\right)$ belong to $\{0,1\}$.
\end{lemma}
\begin{proof}{lemma}{order_extremal}
Let $A=\{(\alpha,\eta,\xi)\in{\mathbf A}\times{\mathbf
X}^2:\,\eta\leq\xi\}$ and assume
$\lambda:=\overline{\nu}(A)\in(0,1)$. Since the coupling
defined by $\overline{L}$ is monotone, 
we have ${\bf
1}_A(\alpha_t,\eta_t,\xi_t)\geq {\bf 1}_A(\alpha_0,\eta_0,\xi_0)$.
By Lemma \ref{stable},
$$\overline{\nu}_A:=\overline{\nu}(d\alpha,d\eta,d\xi|(\alpha,\eta,\xi)\in A)\in\mathcal I_{\overline{L}}$$
From
$
\overline{\nu}=\lambda\overline{\nu}_A+(1-\lambda)\overline{\nu}_{A^c},
$
we deduce $\overline{\nu}_{A^c}\in\mathcal I_{\overline{L}}$. Since
$A$ is shift invariant in ${\mathbf A}\times{\mathbf X}^2$,
$\overline{\nu}_A$ and $\overline{\nu}_{A^c}$ lie in
$\overline{\mathcal{S}}$. By extremality of $\overline{\nu}$, we
must have $\overline{\nu}_A=\overline{\nu}_{A^c}$ which is
impossible since these measures are supported on disjoint sets.
\end{proof}
\mbox{}\\ \\
Attractiveness assumption \textit{(A5)} ensures that an initially
ordered pair of coupled configurations remains ordered at later
times. Assumptions \textit{(A1), (A4)} induce  a stronger property:
pairs of opposite  discrepancies  between two coupled configurations
eventually get killed, so that the two configurations become
ordered.
\begin{proposition}\label{prop_irred}
Every $\overline{\nu}\in{\mathcal I}_{\overline{L}}\cap\overline{\mathcal S}$ is supported on
$\overline{\mathcal O}$.
\end{proposition}
\begin{proof}{proposition}{prop_irred}
We  follow the scheme used in \cite{lig,andjel,coc,guiol,gs} for the non-disordered case,
 and  only sketch the arguments  needed for the disordered setting.
\\ \\
{\em Step 1.} For $x\in\Z$, let $f_x(\eta,\xi)=(\eta(x)-\xi(x))^+$.
By translation invariance of  $\overline{\nu}$,  the  shift  commutation
property \eqref{commutation} and \eqref{coupling_t},
\eqref{xy-gen-coupl},
\begin{eqnarray*}
0&=&\int\overline{L}_\alpha f_0 (\alpha,\eta,\xi)\overline{\nu}(d\alpha,d\eta,d\xi) \\
& = & \sum_{v\in\Z}
\int\overline{L}_\alpha^{0,v} [f_0+f_v] (\alpha,\eta,\xi)\overline{\nu}(d\alpha,d\eta,d\xi)
\end{eqnarray*}
On the other hand (see \cite{coc,gs})
\begin{eqnarray*}
\overline{L}_\alpha^{0,v}(f_0+f_v)& \leq & -p(v)\alpha(0)|b(\eta(0),\eta(v))-b(\xi(0),\xi(v))|\\
& & \times\left(
{\bf 1}_{\{\eta(0)>\xi(0),\,\eta(v)<\xi(v)\}}+{\bf 1}_{\{\eta(0)<\xi(0),\,\eta(v)>\xi(v)\}}
\right)
\end{eqnarray*}
Using Assumptions  \textit{(A4)--(A5)},  \eqref{alpha_bounds} and
translation invariance of $\overline{\nu}$, we obtain
\be\label{opposite_discrepancies}
\begin{array}{ll}
& \overline{\nu}\left(
(\alpha,\eta,\xi):\,\eta(x)>\xi(x),\,\eta(y)<\xi(y)
\right)\\
+ & \overline{\nu}\left(
(\alpha,\eta,\xi):\,\eta(x)<\xi(x),\,\eta(y)>\xi(y)
\right)=0
\end{array}
\ee
for $x\neq y$ with $p(y-x)+p(x-y)>0$. Whenever one of the events in
\eqref{opposite_discrepancies} holds, we say there is
\textit{a pair of opposite discrepancies} at  $(x,y)$.  \\ \\
{\em Step 2.} One proves by induction that, for all $n\in\N$,
\eqref{opposite_discrepancies} holds if $x\neq y$ with
$p^{*n}(y-x)+p^{*n}(x-y)>0$. The induction step is based on the
following idea. Assume $(\eta,\xi)$  has 
 a pair of opposite discrepancies at $(x,y)$. Then one can find
a finite path of coupled transitions (with rates uniformly bounded
below thanks to  \textit{(A4)--(A5)}  and \eqref{alpha_bounds}),
leading to a coupled state with a pair of opposite discrepancies,
either at $(x,z)$  for some $z$  with
$p^{*(n-1)}(z-x)+p^{*(n-1)}(x-z)>0$, or at $(z,y)$ with
$p^{*(n-1)}(y-z)+p^{*(n-1)}(z-y)>0$. This part of the argument is
insensitive to the presence of disorder so long as $\alpha(x)$ is uniformly bounded below.\\ \\
{\em Conclusion.} By irreducibility assumption \textit{(A1)},
\eqref{opposite_discrepancies} holds for all $(x,y)\in\Z^2$ with $x\neq y$.
This implies $\overline{\nu}(\overline{\mathcal O})=1$.
\end{proof}
\mbox{}\\ \\
We are now in a position to prove Proposition \ref{invariant}.\\
\begin{proof}{proposition}{invariant}
We define
$$\mathcal R^Q:=\left\{\int\eta(0)\nu(d\alpha,d\eta):\,
\nu\in\left(\mathcal I_{L}\cap\mathcal S\right)_e,\,\nu\mbox{ has
}\alpha\mbox{-marginal }Q\right\}$$
 Let $\nu^i\in\left(\mathcal I_{L}\cap\mathcal S\right)_e$  with $\alpha$-marginal $Q$ and
$\rho^i:=\int\eta(0)\nu^i(d\alpha,d\eta)\in{\mathcal R}^Q$ for
$i\in\{1,2\}$. Assume $\rho^1\leq\rho^2$.  Using Lemma
\ref{lemma_conditional},\textit{(iii)}, Lemmas
\ref{coupling_extremal} and \ref{order_extremal}, and Proposition
\ref{prop_irred},   we obtain  $\nu^1\ll\nu^2$,  that is
\eqref{ordered_measures}.
Existence \eqref{Rezakhanlou} of an asymptotic particle density can be
obtained by a proof analogous to \cite[Lemma 14]{mrs}, where the
space-time ergodic theorem
is applied to the joint disorder-particle process. Then,
closedness of ${\mathcal R}^Q$ is established as in \cite[Proposition 3.1]{bgrs2}.
We end up proving the weak continuity statement given the rest of
the proposition. Let $\rho,\rho'\in{\mathcal R}^Q$ with
$\rho\leq\rho'$. By \eqref{ordered_measures}  and Lemma
\ref{lemma_conditional}, there exists a coupling
$\overline{\nu}^{Q,\rho,\rho'}(d\alpha,d\eta,d\xi)$ of
$\nu^{Q,\rho}(d\alpha,d\eta)$ and  $\nu^{Q,\rho'}(d\alpha,d\xi)$ 
supported on $\overline{\mathcal O}_+$. Thus, for $x\in\Z$ 
\be\label{standard_coupling}\int|\eta(x)-\xi(x)|\overline{\nu}^{Q,\rho,\rho'}(d\alpha,d\eta,d\xi)=|\rho-\rho'|\ee
from which weak continuity follows by a coupling argument.
\end{proof}
\begin{remark}\label{remark_lipschitz}
Since 
\be\label{compare_fluxes}
G^Q(\rho)-G^Q(\rho')=\int [\widetilde{\jmath}(\alpha,\eta)-\widetilde{\jmath}(\alpha,\xi)]\overline{\nu}^{Q,\rho,\rho'}(d\alpha,d\eta,d\xi)
\ee
a Lipschitz constant $V$ for $G^Q$  follows 
from  \eqref{other_flux}, \eqref{standard_coupling}: 
\be\label{maxspeed}V=2c^{-1}||b||_\infty\sum_{z\in\Z}|z|p(z)\ee
\end{remark}
\section{ Proof of hydrodynamics}
\label{proof_hydro}
In this section, we prove the hydrodynamic limit following the
strategy   
introduced in \cite{bgrs1,bgrs2} and significantly strengthened in \cite{bgrs3}. That is, we reduce general
Cauchy data to step initial conditions (the so-called Riemann
problem)  and use a constructive approach (as in \cite{av}). 
Some technical details similar to \cite{bgrs3}
will be omitted. We shall rather focus  on how to deal with the disorder, which is the substantive
part of this paper.
The measure $Q$ being fixed once and for all by Theorem
\ref{th:hydro}, we simply write   $\nu^\rho$, $\mathcal
R$, $G$.  
\subsection{Riemann problem}\label{riemann}
Let $\lambda,\rho\in[0,K]$ with $\lambda<\rho$ (for $\lambda>\rho$ replace
infimum with supremum below), and
\begin{equation}\label{eq:rie}
R_{\lambda,\rho}(x,0)=\lambda \mathbf 1_{\{x<0\}}+\rho \mathbf
1_{\{x\geq 0\}}
\end{equation}
The entropy solution to  the conservation law  \eqref{hydrodynamics}
with initial condition \eqref{eq:rie}, denoted by
$R_{\lambda,\rho}(x,t)$, is given (\cite[Proposition 4.1]{bgrs2}) by
a variational formula, and satisfies
\begin{eqnarray}
\int_v^w R_{\lambda,\rho}(x,t)dx & = & t[\mathcal
G_{v/t}(\lambda,\rho)-\mathcal G_{w/t}(\lambda,\rho)],\mbox{
with}\cr \mathcal G_v(\lambda,\rho) & := & \inf\left\{
G(r)-vr:\,r\in[\lambda,\rho]\cap\mathcal R
\right\}\label{limiting_current}
\end{eqnarray}
for all $v,w\in\R$.
Microscopic  states with profile \eqref{eq:rie} will be constructed
using the following lemma, established in Subsection
\ref{technical}  below. 
\begin{lemma}\label{big_coupling}
There exist random variables $\alpha$ and
$(\eta^\rho:\,\rho\in\mathcal R)$ on a probability space
$(\Omega_{\mathbf A},\mathcal F_{\mathbf A},\Prob_{\mathbf A})$ such
that
\beq\label{marginals}(\alpha,\eta^\rho)\sim\nu^\rho,&& \alpha\sim
Q\\ \label{strassen} \Prob_{\mathbf A}-a.s.,&&
\rho\mapsto\eta^\rho\mbox{ is nondecreasing}\eeq
\end{lemma}
Let $\overline{\nu}^{\lambda,\rho}$ denote the distribution of
$(\alpha,\eta^\lambda,\eta^\rho)$, and
$\overline{\nu}^{\lambda,\rho}_\alpha$ the conditional distribution of
$(\alpha,\eta^\lambda,\eta^\rho)$ given $\alpha$.
 For $(x_0,t_0)\in\Z\times\R^+$, 
the {\sl space-time shift} $\theta_{x_0,t_0}$
is defined
for any $\omega\in\Omega$, for any $(t,x,z,u)\in\R^+\times\Z\times\Z\times[0,1]$, by
\[
(t,x,z,u)\in\theta_{x_0,t_0}\omega\mbox{ if and only if
}(t_0+t,x_0+x,z,u)\in\omega\]
%
By its definition and  property \eqref{shift_t}, the mapping introduced in \eqref{unique_mapping}
satisfies,  for all
 $s,t\geq 0$, $x\in\Z$ and $(\alpha,\eta,\omega)\in{{\mathbf A}}\times{\mathbf
X}\times{\Omega}$:
\begin{eqnarray*} \label{mapping_markov}
\eta_s(\alpha,\eta_t(\alpha,\eta,\omega),\theta_{0,t}\omega)&=&\eta_{t+s}(\alpha,\eta,\omega)
\\ \label{mapping_shift}
\tau_x\eta_t(\alpha,\eta,\omega)&=&\eta_t(\tau_x\alpha,\tau_x\eta,\theta_{x,0}\omega)
\end{eqnarray*}
We now  introduce an extended shift $\theta'$ on $\Omega'={\mathbf
A}\times{\mathbf X}^2\times\Omega$. If
$\omega'=(\alpha,\eta,\xi,\omega)$
%
denotes a generic element of $\Omega'$, we set
\be\label{extended_shift}\theta'_{x,t}\omega' =
(\tau_x\alpha,\tau_x\eta_t(\alpha,\eta,\omega),\tau_x\eta_t(\alpha,\xi,\omega),\theta_{x,t}\omega)
\ee
 It is important to note that this shift incorporates disorder.
Let $T:{\mathbf X}^2\to{\mathbf X}$ be given by
\be\label{transfo_T}
T(\eta,\xi)(x)=\eta(x){\bf 1}_{\{x< 0\}}+\xi(x){\bf 1}_{\{x\geq
0\}} \ee
The  main result of this  subsection is
\begin{proposition}
\label{corollary_2_2}  Set, for $t\ge 0$,
\be \label{def_empirical_shift}
\beta^N_t(\omega')(dx):=\pi^N(\eta_t(\alpha,T(\eta,\xi),\omega))(dx)
\ee
For all $t>0$, $s_0\geq 0$ and $x_0\in\R$, we have that, for
$Q$-a.e. $\alpha\in{\mathbf A}$,
\[
\lim_{N\to\infty}\beta^N_{Nt}(\theta'_{\lfloor Nx_0\rfloor ,Ns_0}\omega')(dx)=R_{\lambda,\rho}(.,t)dx,
\quad\overline{\nu}_\alpha^{\lambda,\rho}\otimes\Prob\mbox{-a.s.}
\]
\end{proposition}
 Proposition \ref{corollary_2_2} will follow from a law of large
numbers for currents.
Let $x_.=(x_t,\,t\geq 0)$ be a $\Z$-valued {\em cadlag} random path,
with $\abs{x_t-x_{t^-}}\leq 1$, independent of the Poisson measure
$\omega$.
We define the particle current seen by an observer travelling along
this path by
\be \label{current_3} \varphi^{x_.}_t(\alpha,\eta_0,\omega)
=\varphi^{x_.,+}_t(\alpha,\eta_0,\omega)
-\varphi^{x_.,-}_t(\alpha,\eta_0,\omega)+\widetilde{\varphi}^{x_.}_t(\alpha,\eta_0,\omega)
\ee
where
$\varphi^{x_.,\pm}_t(\alpha,\eta_0,\omega)$ count the number of
rightward/leftward crossings  of $x_.$ due to particle
jumps, and
$\widetilde{\varphi}^{x_.}_t(\alpha,\eta_0,\omega)$ is the current due
to the self-motion of the observer.
We shall  write  $\varphi^v_t$ in the particular case $x_t=\lfloor
vt\rfloor$.
Set $\phi^{v}_t(\omega'):=\varphi^{v}_t(\alpha,T(\eta,\xi),\omega)$.
Note that  for $(v,w)\in\R^2$,
$\beta^N_{Nt}(\omega')([v,w])=t(Nt)^{-1}(\phi^{v/t}_{Nt}(\omega')-\phi^{w/t}_{Nt}(\omega'))$.
By \eqref{limiting_current},  Proposition \ref{corollary_2_2} is
reduced to  
\begin{proposition}
\label{proposition_2_2}
For all $t>0$, $a\in\R^+,b\in\R$ and  $v\in\R$,
\be\lim_{N\to\infty}(Nt)^{-1}\phi^{v}_{Nt}(\theta'_{\lfloor b N\rfloor ,a N}\omega') =
\mathcal G_v(\lambda,\rho)\qquad\overline{\nu}^{\lambda,\rho}\otimes\Prob-a.s. \label{as}\ee
\end{proposition}
To prove Proposition \ref{proposition_2_2}, we introduce a
probability space $\Omega^+$, whose generic element is denoted by
$\omega^+$, on which is defined a Poisson process
$(N_t(\omega^+))_{t\ge 0}$ with intensity $\abs{v}$ ($v\in\R$).
Denote by ${\Prob}^+$ the associated probability. Set
\begin{eqnarray}\label{def_poisson}
x_s^N(\omega^+)&:=&({\rm sgn}(v))\left[N_{a N+s}(\omega^+)-N_{aN}(\omega^+)\right]\\
\label{mapping_tilde}\widetilde{\eta}^N_s(\alpha,\eta_0,\omega,\omega^+)
&:=&\tau_{x_s^N(\omega^+)}\eta_s(\alpha,\eta_0,\omega)\\
\label{mapping_tilde_alpha}
\widetilde{\alpha}^N_s(\alpha,\omega^+) & := & \tau_{x^N_s(\omega^+)}\alpha
\end{eqnarray}
Thus  $(\widetilde{\alpha}_s^N,\widetilde{\eta}_s^N)_{s\ge 0}$ is a
Feller process with generator
\[
 L^v=L+S^v,\quad S^v f(\alpha,\zeta)=\abs{v}
[f(\tau_{{\rm sgn}(v)}\alpha,\tau_{{\rm
sgn}(v)}\zeta)-f(\alpha,\zeta)] \]
 for $f$ local and $\alpha\in\mathbf A$, $\zeta\in{\mathbf X}$.
Since any translation invariant measure on ${\mathbf
A}\times{\mathbf X}$ is stationary for the pure shift generator
$S^v$, we have
${\mathcal I}_L\cap{\mathcal S}={\mathcal
I}_{L^v}\cap{\mathcal S}$.
Define the time and space-time empirical measures (where $\eps>0$) by
\begin{eqnarray} \label{def_empirical} m_{tN}(\omega',\omega^+)&:=&(Nt)^{-1}\int_0^{tN}
\delta_{(\widetilde{\alpha}^N_s(\alpha,\omega^+),\widetilde{\eta}^N_s(\alpha,T(\eta,\xi),\omega,\omega^+))}ds\\
\label{def_empirical_2}
m_{tN,\eps}(\omega',\omega^+)&:=&\abs{\Z\cap[-\eps N,\eps
N]}^{-1}\sum_{x\in\Z:\,\abs{x}\leq \eps
N}\tau_xm_{tN}(\omega',\omega^+) 
\end{eqnarray}
Notice  that there is a disorder component we cannot omit in the empirical measure, 
although ultimately  we are only interested
in the behavior of the $\eta$-component.
Let ${\mathcal M}_{\lambda,\rho}$ denote
the compact set   of probability measures $\mu(d\alpha,d\eta)\in \mathcal I_{L}\cap\mathcal S$ such that
$\mu$ has $\alpha$-marginal $Q$, and
$\nu^\lambda\ll\mu\ll\nu^\rho$. By Proposition \ref{invariant},
\be\label{setofmeasures}
\mathcal M_{\lambda,\rho}=\left\{\nu(d\alpha,d\eta)
=\int\nu^r(d\alpha,d\eta)\gamma(dr):\,\gamma\in\mathcal P([\lambda,\rho]\cap\mathcal R)\right\}
\ee
The key ingredients for Proposition \ref{proposition_2_2} are the
following lemmas, proved in Subsection \ref{technical}  below.
\begin{lemma}\label{current_comparison}
The function $\phi^v_t(\alpha,\eta,\xi,\omega)$ is increasing in $\eta$, decreasing in $\xi$.
\end{lemma}
\begin{lemma}\label{lemma_empirical}
With
$\overline{\nu}^{\lambda,\rho}\otimes{\Prob}\otimes{\Prob}^+$-probability
one, every subsequential limit as $N\to\infty$ of
$m_{tN,\eps }(\theta'_{\lfloor b N\rfloor ,a
N}\omega',\omega^+)$ lies in $\mathcal M_{\lambda,\rho}$.
\end{lemma}
\begin{proof}{Proposition}{proposition_2_2}
We will show that
\beq \liminf_{N\to\infty}\,(Nt)^{-1}\phi^{v}_{tN}\circ\theta'_{\lfloor b N\rfloor ,a
N}(\omega') & \geq & \mathcal G_v(\lambda,\rho),\quad
\overline{\nu}^{\lambda,\rho}\otimes\Prob\mbox{-a.s.}\label{liminf_better}\\
\label{limsup}
\limsup_{N\to\infty}\,(Nt)^{-1}\phi^{v}_{tN}\circ\theta'_{\lfloor b N\rfloor ,a
N}(\omega') & \leq & \mathcal G_v(\lambda,\rho),\quad \overline{\nu}^{\lambda,\rho}\otimes\Prob\mbox{-a.s.}
\eeq
{\em Step one: proof of \eqref{liminf_better}.}\par\noindent
  Setting
$\varpi_{a N}=\varpi_{a
N}(\omega'):=T\left( \tau_{\lfloor b N\rfloor }\eta_{a
N}(\alpha,\eta,\omega), \tau_{\lfloor b N\rfloor }\eta_{a N}(\alpha,\xi,\omega)
 \right)$,
we have
\be \label{sothat} (Nt)^{-1}\phi^v_{tN}\circ\theta'_{\lfloor b N\rfloor ,a
N}(\omega')=(Nt)^{-1}\varphi^{v}_{tN}(\tau_{\lfloor b
N\rfloor }\alpha,\varpi_{a N},\theta_{\lfloor b
N\rfloor ,a N}\omega) \ee
Let,   for every $(\alpha,\zeta,\omega,\omega^+)\in{\mathbf
A}\times{\mathbf X}\times\Omega\times\Omega^+$ and $x^N_.(\omega^+)$
given by \eqref{def_poisson},
\be\label{psi-t-v}
\psi_{tN}^{v,\eps}(\alpha,\zeta,\omega,\omega^+):=\abs{\Z\cap[-\eps
N,\eps N]}^{-1}\sum_{y\in\Z:\,\abs{y}\leq\eps
N}\varphi^{x^N_.(\omega^+)+y}_{tN}(\alpha,\zeta,\omega) \ee
Note that  $\lim_{N\to\infty}(Nt)^{-1}x_{tN}^N(\omega^+)= v$,  $\Prob^+$-a.s., and that for
two paths $y_.,z_.$  (see \eqref{current_3}),
$$
\abs{\varphi^{y_.}_{tN}(\alpha,\eta_0,\omega)-\varphi^{z_.}_{tN}(\alpha,\eta_0,\omega)} \leq
K\left(\abs{y_{tN}-z_{tN}}+ \abs{y_0-z_0}\right)
$$
Hence the proof of \eqref{liminf_better} reduces to that of the same
inequality where we replace $(Nt)^{-1}\phi^{v}_{tN}\circ\theta'_{\lfloor b
N\rfloor ,a N}(\omega')$ by
$ (Nt)^{-1}\psi^{v,\eps}_{tN}(\tau_{\lfloor b N\rfloor }\alpha,\varpi_{a
N},\theta_{\lfloor b N\rfloor ,a N}\omega,\omega^+)$ 
and $\overline{\nu}^{\lambda,\rho}\otimes\Prob$ by
$\overline{\nu}^{\lambda,\rho}\otimes\Prob\otimes\Prob^+$.
By definitions \eqref{def_f}, \eqref{current_3} of flux and current,
for any $\alpha\in\mathbf{A}$, $\zeta\in\mathbf{X}$,
\begin{eqnarray*} &&M^{x,v}_{tN}(\alpha,\zeta,\omega,\omega^+):=
\varphi^{x^N_.(\omega^+)+x}_{tN}(\alpha,\zeta,\omega)-\nonumber\\
&& \int_0^{tN} \tau_x\left\{
j(\widetilde{\alpha}^N_{s}(\alpha,\omega^+),\widetilde{\eta}^N_{s}(\alpha,\zeta,\omega,\omega^+))
-v (\widetilde{\eta}^N_{s}(\alpha,\zeta,\omega,\omega^+))({\bf
1}_{\{v>0\}})\right\}ds\label{martingale}
\end{eqnarray*}
is a mean $0$ martingale under $\Prob\otimes\Prob^+$.  Let
\beq\nonumber R_{tN}^{\eps,v}&:=&\label{as_2} \left(Nt\abs{\Z\cap[-\eps
N,\eps N]}\right)^{-1}\sum_{x\in\Z:\,\abs{x}\leq\eps N}M^{x,v}_{tN}(
\tau_{\lfloor b N\rfloor }\alpha,\varpi_{a N},\theta_{\lfloor b
N\rfloor ,a N}\omega,\omega^+)\\\nonumber
&=&(Nt)^{-1}\psi^{v,\eps}_{tN}(\tau_{\lfloor b N\rfloor }\alpha,\varpi_{a
N},\theta_{\lfloor b N\rfloor ,a N}\omega,\omega^+)\\
&&-\int [j(\alpha,\eta)-v\eta({\bf
1}_{\{v>0\}})]m_{tN,\eps}(\theta'_{\lfloor b N\rfloor ,a
N}\omega',\omega^+)(d\alpha,d\eta)
\label{replace_oncemore} \eeq
where the last equality comes from \eqref{def_empirical_2},
\eqref{psi-t-v}.
 The exponential martingale associated with $M^{x,v}_{tN}$ 
yields a Poissonian bound, uniform in $(\alpha,\zeta)$, for the
exponential moment of $M_{tN}^{x,v}$ with respect to
$\Prob\otimes\Prob^+$. Since $\varpi_{a N}$ is independent of
$(\theta_{\lfloor b N\rfloor ,a N}\omega,\omega^+)$ under
$\overline{\nu}^{\lambda,\rho}\otimes\Prob\otimes\Prob^+$, the bound
is also valid under this measure, and Borel-Cantelli's lemma implies
$\lim_{N\to\infty}R_{tN}^{\eps,v}=0$. 
{}From \eqref{replace_oncemore}, Lemma \ref{lemma_empirical} and
Corollary \ref{corollary_invariant}, \textit{(ii)} imply
\eqref{liminf_better}, as well as
\be \limsup_{N\to\infty}\,(Nt)^{-1}\phi^{v}_{tN}\circ\theta'_{\lfloor b N\rfloor ,a
N}(\omega')  \leq  \sup_{r\in[\lambda,\rho]\cap{\mathcal R}}
[G(r)-v r],\quad
\overline{\nu}^{\lambda,\rho}\otimes\Prob\mbox{-a.s.}\label{liminf_better_2}\ee
{\em Step two: proof of \eqref{limsup}.}
Let $r\in[\lambda,\rho]\cap{\mathcal R}$. We define
$\overline{\nu}^{\lambda,r,\rho}$ as the distribution of
$(\alpha,\eta^\lambda,\eta^r,\eta^\rho)$.
By \eqref{liminf_better} and \eqref{liminf_better_2}, 
\[
\lim_{N\to\infty}\,(Nt)^{-1}\phi^{v}_{tN}\circ\theta'_{\lfloor b N\rfloor ,a
N}(\alpha,\eta^r,\eta^r,\omega)  =  G(r)-vr\]
By Lemma \ref{current_comparison},
\[
\phi^v_{tN}\circ\theta'_{\lfloor bN\rfloor ,aN}(\omega')\leq\phi^v_{tN}\circ\theta'_{\lfloor bN\rfloor ,aN}(\alpha,\eta^r,\eta^r,\omega)\]
 The result follows by continuity of $G$ and minimizing over $r$.
\end{proof}
\subsection{Cauchy problem}
\label{Cauchy}
For two measures $\mu,\nu\in{\mathcal M}^+(\R)$ with compact
support, we define
\be\label{def_delta}\Delta(\mu,\nu):=\sup_{x\in\R}\abs{
\nu((-\infty,x])-\mu((-\infty,x])}\ee
 which satisfies: \emph{(P1)}
 For a sequence  $(\mu_n)_{n\ge 0}$  of measures with uniformly bounded support,
 $\mu_n\to \mu$ vaguely is equivalent to
$\lim_{n\to\infty}\Delta(\mu_n,\mu)=0$;
\emph{(P2)} the macroscopic stability property   (\cite{bm, mrs})
states that $\Delta$ is, with high probability, an ``almost''
nonincreasing function of two coupled particle systems; \emph{(P3)}
correspondingly, there is $\Delta$-stability  for
\eqref{hydrodynamics}, that is,  $\Delta$ is a nonincreasing
function along two entropy solutions (\cite[Proposition 4.1,
\textit{iii), b)}]{bgrs3}).  
\begin{proposition}\label{hydro_finite}
Assume $(\eta^N_0)$ is a sequence of configurations such that:
(i) there exists $C>0$  such that  for all $N\in\N$, $\eta^N_0$ is
supported on $\Z\cap[-CN,CN]$;
(ii) $\pi^N(\eta^N_0)\to u_0(.)dx$ as $N\to\infty$, where $u_0$
 has compact support,  is a.e. ${\mathcal R}$-valued and
has finite space variation.
Let $u(.,t)$ denote the unique entropy solution to \eqref{hydrodynamics} with Cauchy datum $u_0(.)$.
Then,  $Q\otimes\Prob$-a.s. as $N\to\infty$,
\[
\Delta^N(t):=\Delta(\pi^N(\eta^N_{Nt}(\alpha,\eta^N_0,\omega)),u(.,t)dx)
\]
converges uniformly to $0$ on $[0,T]$ for every $T>0$.
\end{proposition}
Theorem \ref{th:hydro} follows  for general initial data $u_0$
by coupling and approximation arguments (see  \cite[Section 4.2.2]{bgrs3}). \\ \\
\begin{proof}{proposition}{hydro_finite}
By initial assumption \eqref{initial_profile_vague},
$\lim_{N\to\infty}\Delta^N(0)=0$.
Let $\eps>0$, and  $\eps'=\eps/(2V)$, for  $V$  given by
\eqref{maxspeed}. Set $t_k=k\eps'$ for $k\le \kappa:=\lfloor
T/\eps'\rfloor $, $t_{\kappa+1}=T$.  Since the number of steps is
proportional to $\eps$, if we want to bound the total error, the
main step is to prove
\be\label{wearegoing_1}
\limsup_{N\to\infty}\sup_{k=0,\ldots,{\mathcal
K}-1}\left[\Delta^N(t_{k+1})-\Delta^N(t_k)\right]\leq
3\delta\eps,\quad Q\otimes\Prob\mbox{-a.s.} \ee
where  $\delta:=\delta(\eps)$  goes to 0 as $\eps$ goes to 0;
the gaps between discrete times are filled by an  
  estimate for the time
modulus of continuity of $\Delta^N(t)$ (see  \cite[Lemma 4.5]{bgrs3}).  \\ \\
{\em Proof of \eqref{wearegoing_1}}. Since $u(.,t_k)$ has locally
finite variation, by \cite[Lemma 4.2]{bgrs3}, for all $\eps>0$
 we can find  
 functions
\be \label{decomp_approx} v_k=\sum_{l=0}^{l_k}r_{k,l}{\bf
1}_{[x_{k,l},x_{k,l+1})} \ee
with
$ -\infty=x_{k,0}<x_{k,1}<\ldots<x_{k,l_k}<x_{k,l_k+1}=+\infty$,
$r_{k,l}\in{\mathcal R}$, $r_{k,0}=r_{k,l_k}=0$,
such that $x_{k,l}-x_{k,l-1}\geq \eps$, and 
\be\label{uniform_approx}
\Delta(u(.,t_k)dx,v_kdx)  \leq  \delta\eps
\ee
 For $t_k\leq t< t_{k+1}$, we denote by $v_k(.,t)$ the entropy
solution to \eqref{hydrodynamics} at time $t$ with Cauchy datum
$v_k(.)$.
 The configuration $\xi^{N,k}$
defined on $(\Omega_{\mathbf A}\otimes\Omega,\mathcal F_{\mathbf
A}\otimes\mathcal F,\Prob_{\mathbf A}\otimes\Prob)$ (see Lemma
\ref{big_coupling}) by
\[
\xi^{N,k}(\omega_{\mathbf
A},\omega)(x):=\eta_{Nt_k}(\alpha(\omega_{\mathbf
A}),\eta^{r_{k,l}}(\omega_{\mathbf A}), \omega)(x),\, \mbox{ if }\,
\lfloor Nx_{k,l}\rfloor \leq x<\lfloor Nx_{k,l+1}\rfloor  \]
is a microscopic version of $v_k(.)$, since by Proposition
\ref{corollary_2_2} with $\lambda=\rho=r^{k,l}$,
\be \label{profile_xi}
\lim_{N\to\infty}\pi^N(\xi^{N,k}(\omega_{\mathbf A},\omega))(dx)
=v_k(.)dx,\quad\Prob_{\mathbf A}\otimes\Prob\mbox{-a.s.} \ee
 We denote by
$\xi^{N,k}_t(\omega_{\mathbf A},\omega)  =
\eta_{t}(\alpha(\omega_{\mathbf A}),\xi^{N,k}(\omega_{\mathbf
A},\omega), \theta_{0,Nt_k}\omega)$
evolution starting  from $\xi^{N,k}$.  
By triangle inequality,
\beq
\Delta^N(t_{k+1})-\Delta^N(t_k) & \leq & \Delta\left[ \pi^N(
\eta^N_{Nt_{k+1}} ),\pi^N( \xi^{N,k}_{N\eps'} )
\right]-\Delta^N(t_k)\label{decomp_delta_1}\\
& + & \Delta\left[ \pi^N( \xi^{N,k}_{N\eps'}
),v_k(.,\eps')dx \right]\label{decomp_delta_2}\\
& + & \Delta(v_k(.,\eps')dx,u(.,t_{k+1})dx)\label{decomp_delta_3}
\eeq
 To conclude, we rely on Properties \emph{(P1)--(P3)} of
$\Delta$:  Since  $\eps'=\eps/(2V)$,  finite propagation property
for \eqref{hydrodynamics} and for the particle system  (see
\cite[Proposition 4.1, \emph{iii), a)} and Lemma 4.3]{bgrs3}) and
Proposition \ref{corollary_2_2} imply
$$
\lim_{N\to\infty}\pi^N(
\xi^{N,k}_{N\eps'}(\omega_{\mathbf A},\omega)
)=v_k(.,\eps')dx,\qquad \Prob_A\otimes\Prob\mbox{-a.s.}
$$
Hence, the term \eqref{decomp_delta_2} converges a.s. to $0$ as
$N\to\infty$. By $\Delta$-stability for \eqref{hydrodynamics}, the
term \eqref{decomp_delta_3}  is bounded by 
$\Delta(v_k(.)dx,u(.,t_k)dx)\leq\delta\eps$.
We now consider the term \eqref{decomp_delta_1}.
By macroscopic stability  (\cite[Theorem 2, Equation (4) and Remark
1]{mrs}), outside probability $e^{-CN\delta\eps}$,
\be\label{macrostablast}\Delta\left[ \pi^N(
\eta^N_{Nt_{k+1}}),\pi^N( \xi^{N,k}_{N\eps'} ) \right]\leq
\Delta\left[ \pi^N( \eta^N_{Nt_{k}} ),\pi^N( \xi^{N,k} )
\right]+\delta\eps
\ee
Thus the event \eqref{macrostablast} holds a.s. for $N$ large
enough.
By triangle inequality,
\begin{eqnarray*}
&&\Delta\left[ \pi^N( \eta^N_{Nt_{k}} ),\pi^N( \xi^{N,k} ) \right] -
\Delta^N(t_k) \\
&\leq&\Delta\left(u(.,t_k)dx,v_k(.)dx\right)+\Delta\left[v_k(.)dx,\pi^N(\xi^{N,k})\right]
\end{eqnarray*}
for which \eqref{uniform_approx}, \eqref{profile_xi} yield  as
$N\to\infty$ an upper bound $2\delta\varepsilon$, hence
$3\delta\varepsilon$ for the term \eqref{decomp_delta_1}.
\end{proof}
\subsection{Proofs of lemmas}
\label{technical}
\begin{proof}{lemma}{big_coupling}
Let $\mathcal R_d$ be a countable dense subset of $\mathcal R$ that
contains all the isolated points of $\mathcal R$. We denote by
$\mathcal R_d^+$, resp. $\mathcal R_d^-$, the set of $\rho\in[0,K]$
that lie in the closure of $[0,\rho)\cap\mathcal R_d$, resp.
$(\rho,K]\cap\mathcal R_d$.
 Because $\mathcal R$ is closed, we have
${\mathcal R}={\mathcal R_d}\cup{\mathcal R_d^+}\cup{\mathcal
R_d^-}$.
By \eqref{ordered_measures} there exists a subset ${\mathbf A}'$ of
${\mathbf A}$ with $Q$-probability $1$, such that
$\nu^\rho_\alpha\leq\nu^{\rho'}_\alpha$ for all $\alpha\in{\mathbf
A}'$ and $\rho,\rho'\in{\mathcal R}_d$.
By \cite[Theorem 6]{kk}, for every $\alpha\in{\mathbf A}'$, there
exists a family of random variables
$(\eta^\rho_\alpha:\,\rho\in{\mathcal R}_d)$ on a probability space
$(\Omega_\alpha,\mathcal F_\alpha,\Prob_\alpha)$, such that
\eqref{marginals}--\eqref{strassen} hold for $\rho\in{\mathcal
R}_d$.
Let $\Omega_{\mathbf A}=\{(\alpha,\omega_\alpha):\,\alpha\in{\mathbf
A'},\,\omega_\alpha\in\Omega_\alpha\}$, $\mathcal F_{\mathbf A}$ be
the $\sigma$-field generated by mappings
$(\alpha,\omega_\alpha)\mapsto\eta^\rho(\alpha,\omega_\alpha):=\eta_\alpha^\rho(\omega_\alpha)$
for $\rho\in{\mathcal R}_d$, and $\Prob_{\mathbf
A}(d\alpha,d\omega_\alpha)=Q(d\alpha)\otimes\Prob_\alpha(d\omega_\alpha)$.
Now consider $\rho\in{\mathcal R}\setminus{\mathcal R}_d$. Since
$\eta^r_\alpha$ is a nondecreasing function of $r$, for every
$\alpha\in{\mathbf A}'$ and $\omega_\alpha\in\Omega_\alpha$,
$\eta^{\rho+}(\alpha,\omega_\alpha):=\lim_{r\to\rho,r<\rho,r\in{\mathcal
R}_d}\eta^r_\alpha(\omega_\alpha)$ exists if $\rho\in{\mathcal
R}_d^+$, and
$\eta^{\rho-}(\alpha,\omega_\alpha):=\lim_{r\to\rho,r>\rho,r\in{\mathcal
R}_d}\eta^r_\alpha(\omega_\alpha)$ exists if $\rho\in{\mathcal
R}_d^-$.
We set
$\eta^\rho(\alpha,\omega_\alpha)=\eta^{\rho+}(\alpha,\omega_\alpha)$
if $\rho\in{\mathcal R}_d^+$,
$\eta^\rho(\alpha,\omega_\alpha)=\eta^{\rho-}(\alpha,\omega_\alpha)$
otherwise. Suppose for instance $\rho\in\mathcal R_d^+$. Since
$\eta^{\rho+}$ is a $\Prob_{\mathbf A}$-a.s. limit of  $\eta^r$ as
$r\to\rho$, $r<\rho$, $r\in\mathcal R_d$, it is a limit in
distribution. Weak continuity of $\nu^\rho$ then implies
\eqref{marginals}. Property \eqref{strassen} on $\mathcal R$ follows
from the property on $\mathcal R_d$ and definitions of
$\eta^{\rho\pm}$.
\end{proof}
\mbox{}\\ \\
To prove Lemma \ref{lemma_empirical}, we need the following  uniform
upper bound (proved in \cite[Lemma 3.4]{bgrs3}).
 \begin{lemma}\label{deviation_empirical}
Let ${\bf P}_\nu^v$ denote the law of a Markov process $(\widetilde\alpha_.,\widetilde\xi_.)$ with
generator $L^v$ and initial distribution $\nu$. For $\eps>0$, let
\be \label{def_empirical-gen}
\pi_{t,\eps}:=\abs{\Z\cap[-\eps t,\eps
t]}^{-1}\sum_{x\in\Z\cap[-\eps t,\eps t]}t^{-1}\int_0^t\delta_{(\tau_x\widetilde\alpha_s,\tau_x\widetilde\xi_s)}ds
\ee
Then, there exists a functional ${\mathcal D}_v$ which is
nonnegative, l.s.c., and satisfies
${\mathcal D}_v^{-1}(0)={\mathcal I}_{L^v}$, such that,
for every closed subset $F$ of ${\mathcal P}({\bf A}\times\bf
X)$,
\be \label{ld} \limsup_{t\to\infty}t^{-1}\log\sup_{\nu\in\mathcal P({\bf
A}\times{\mathbf X})}{\bf
P}_\nu^v\left(\pi_{t,\eps}(\widetilde\xi_.)\in F\right)\leq
-\inf_{\mu\in F}{\mathcal D}_v(\mu) \ee
\end{lemma}
\begin{proof}{Lemma}{lemma_empirical}
We give a brief sketch of the arguments  (details are similar to
\cite[Lemma 3.3]{bgrs3}).  Spatial averaging in
\eqref{def_empirical-gen} implies that any subsequential limit
$\mu$ lies in $\mathcal S$. Lemma \ref{deviation_empirical} and
Borel-Cantelli's Lemma imply that $\mu$ lies in $\mathcal I_{L^v}$
(uniformity in \eqref{ld} is important because $\theta$-shifts make
the initial distribution of the process unknown). Finally, the
inequality $\nu^\lambda\ll\mu\ll\nu^\rho$ is obtained by coupling
the initial distribution with $\eta^\lambda$ and $\eta^\rho$, using
attractiveness and space-time ergodicity for the equilibrium
processes.
\end{proof}
\mbox{}\\ \\
\begin{proof}{lemma}{current_comparison}
Assume for instance $\eta\leq\eta'$.
Let $\gamma:=T(\eta,\xi)$ and $\gamma':=T(\eta',\xi)$,
$\gamma_t=\eta_t(\alpha,\gamma,\omega)$ and
$\gamma'_t=\eta_t(\alpha,\gamma',\omega)$.
By \eqref{attractive_0}, $\gamma_t\leq\gamma'_t$
for all $t\geq 0$. By definition of the current,
$\phi^v_t(\alpha,\eta',\xi,\omega)-\phi^v_t(\alpha,\eta,\xi,\omega)
=\sum_{x>vt}[\gamma'_t(x)-\gamma_t(x)]\geq 0$.
\end{proof}
\section{Other models}\label{sec:general}
For the proof of Theorem \ref{th:hydro} we have not used the particular
form of $L_\alpha$ in \eqref{generator}, but the following  properties.\\ \\
1) The set of environments is a probability space $({\mathbf
A},{\mathcal F}_{\mathbf A},Q)$, where $\mathbf A$ is a  compact
metric space and ${\mathcal F}_{\mathbf A}$ its Borel
$\sigma$-field. On $\mathbf A$ we have a group of space shifts
$(\tau_x:\,x\in\Z)$, with respect to which $Q$ is ergodic.  For each
$\alpha\in{\mathbf A}$, $L_\alpha$ is the generator of a Feller
process on ${\mathbf X}$ 
that satisfies \eqref{commutation}.
The latter should be viewed as the assumption on ``how the
disorder enters the dynamics''. It is equivalent to  $L$ satisfying \eqref{eq:L-transl-inv}, 
that is  being a translation-invariant generator on ${\mathbf A}\times{\mathbf X}$.\\ \\
2) For $L_\alpha$ we can define a graphical construction  \eqref{update_rule}
on a space-time Poisson space $(\Omega,\mathcal F,\Prob)$ such that
$L_\alpha$ coincides with  \eqref{gengen},  for  some 
mapping ${\mathcal T}^{\alpha,z,v}$ satisfying the shift commutation
and strong attractiveness properties \eqref{shift_t} and
\eqref{attractive_0}.
The existence of this graphical construction for the infinite-volume
system follows from assumption \textit{(A2)}, which controls the rate of faraway jumps.
 This assumption is also responsible for the finite propagation property of discrepancies in the particle
system, and its macroscopic counterpart, the Lipschitz continuity of the flux function 
 (see \eqref{flux}, 
Remarks \ref{remark_flux} and \ref{remark_lipschitz}). \\ \\
3) Irreducibility and non-degeneracy assumptions  \textit{(A1), (A4)}
(combined with attractiveness  assumption  \textit{(A5)}) imply Proposition \ref{prop_irred}. \\ \\
In the sequel we consider other models satisfying 1) and 2), for
which appropriate assumptions replacing   \textit{(A1)--(A5)} 
imply existence of  a graphical construction, and Proposition
\ref{prop_irred}  as in 3).  In these examples, the transition defined by  $\mathcal T^{\alpha,z,v}$ in \eqref{update_rule}
is a particle jump, that is of the form 
$\mathcal T^{\alpha,z,v}\eta=\eta^{x(\alpha,z,v),y(\alpha,z,v)}$. It follows that \eqref{gengen} yields (in replacement of \eqref{generator})
\be\label{generic_form}
L_\alpha f(\eta)=\sum_{x,y\in{\Z}}c_\alpha(x,y,\eta)\left[ f\left(\eta^{x,y} \right)-f(\eta)
\right] 
\ee
where
\be\label{general_rate}
c_\alpha(x,y,\eta)=\sum_{z\in\Z}m\left(
\left\{
v\in\mathcal V:\,\mathcal T^{\alpha,z,v}\eta=\eta^{x,y}
\right\}
\right)
\ee
and the shift-commutation property \eqref{shift_t} implies
\be\label{shift_c}
c_\alpha(x,y,\eta)=c_{\tau_x\alpha}(0,y-x,\tau_x\eta)
\ee
which, for \eqref{generic_form}, is equivalent to \eqref{commutation}.
Microscopic fluxes \eqref{def_f} and \eqref{other_flux} more generally write
\beq\nonumber
j^+(\alpha,\eta) &=& \sum_{y,z\in\Z:\,y\leq 0<y+z}
c_\alpha(\eta(y),\eta(y+z))\\\nonumber
\nonumber j^-(\alpha,\eta) & = &
\sum_{y,z\in\Z:\,y+z\leq 0<y}c_\alpha(\eta(y),\eta(y+z)) \\
\widetilde{\jmath}(\alpha,\eta) & = & \sum_{z\in\Z}zc_\alpha(0,z,\eta)\label{other_flux_general}
\eeq
\subsection{Generalized misanthropes' process}
Let $c\in(0,1)$, and  $p(.)$ (resp. $P(.)$), be a probability
distribution on $\Z$ satisfying assumption \textit{(A1)} (resp.
\textit{(A2)}).  Define $\mathbf A$ to be the set of  functions
 $B:\Z^2\times\{0,\ldots,K\}^2\to\R^+$ such that  for all
$(x,z)\in\Z^2$,  $B(x,z,.,.)$ satisfies assumptions
\textit{(A3)--(A5)} and
\begin{eqnarray}\label{genmis_1}
B(x,z,1,K-1) & \geq & cp(z)\\
\label{genmis_2}
\quad B(x,z,K,0) & \leq & c^{-1}P(z)
\end{eqnarray}
The shift operator  $\tau_y$  on $\mathbf A$ is defined by
$
(\tau_y B)(x,z,n,m)=B(x+y,z,n,m)
$.
We generalize \eqref{generator} by setting
\be\label{generator_genmis}
L_\alpha f(\eta)=\sum_{x,y\in{\Z}}B(x,y-x,\eta(x),\eta(y)) \left[ f\left(\eta^{x,y} \right)-f(\eta)
\right]
\ee
where we assume that the distribution $Q$ of $B(.,.,.,.)$ is ergodic
with respect to the above spatial shift  (we kept the notation $L_\alpha$
to be consistent with the rest of the paper, but we should have written $L_B$). 
Assumption \eqref{genmis_1} replaces \textit{(A1)} and implies
Proposition \ref{prop_irred}.
Assumption \eqref{genmis_2} replaces \textit{(A2)} and implies existence
of the infinite volume dynamics given by the following graphical construction.
For $v=(z,u)$, set  $m(dv)=c^{-1}P(dz)\lambda_{[0,1]}(du)$
in \eqref{special_choice}, and replace \eqref{update_misanthrope} with
\be\label{update_genmis}
{\mathcal T}^{\alpha,x,v}\eta=\left\{
\ba{lll}
\eta^{x,x+z} & \mbox{if} & \displaystyle{ u<\frac{B(x,z,\eta(x),\eta(x+z))}
{c^{-1}P(z)}}\\
\eta &  & \mbox{otherwise}
\ea
\right.
\ee
 Here the microscopic flux \eqref{other_flux_general} writes
$$
\widetilde{\jmath}(\alpha,\eta)=\sum_{z\in\Z}zB(0,z,\eta(0),\eta(z))
$$
and the Lipschitz constant  $V=2c^{-1}\sum_{z\in\Z}|z|P(z)$ for $G^Q$ follows as in \eqref{maxspeed} from  \eqref{standard_coupling}--\eqref{compare_fluxes}.
The basic model \eqref{generator} is  recovered with 
$B(x,z,n,m)=\alpha(x)p(z)b(n,m)$, for $p(.)$ a probability
distribution on $\Z$ satisfying \textit{(A1)--(A2)}, $\alpha(.)$ an
ergodic $(c,1/c)$-valued random field, and $b(.,.)$ a function
satisfying   \textit{(A3)--(A5)}. In this case
\eqref{genmis_1}--\eqref{genmis_2} hold
 with $P(.)=p(.)$. Here are two other examples.\\ \\
 {\em Example 1.1.}  This is the bond-disorder  version of \eqref{generator}:  we have 
$B(x,z,n,m)=\alpha(x,x+z)b(n,m)$,  where  $\alpha=(\alpha(x,y):\,x,y\in\Z)$
is a positive random field on $\Z^2$,  bounded away from $0$, ergodic
with respect to the space shift $\tau_z\alpha=\alpha(.+z,.+z)$. Sufficient
assumptions replacing \textit{(A1)} and \textit{(A2)} are
\be\label{assumptions_bond}c\,p(y-x)\leq \alpha(x,y)\leq c^{-1}P(y-x)\ee
for some constant $c>0$, and probability distributions $p(.)$ and $P(.)$ on $\Z$,
respectively  satisfying  \textit{(A1)} and \textit{(A2)}.\\ \\
 {\em Example 1.2.}   This is a model that switches between two rate functions
according to the environment:  we have
$B(x,z,n,m)=p(z)[(1-\alpha(x))b_0(n,m)+\alpha(x)b_1(n,m)]$,  where
$(\alpha(x),\,x\in\Z)$ is an ergodic $\{0,1\}$-valued field, $p(.)$
satisfies assumption \textit{(A1}), and $b_0$, $b_1$  
assumptions \textit{(A3)--(A5)}.
\subsection{Generalized $k$-step $K$-exclusion process}\label{subsec:k-step}
 We  first recall the definition of the $k$-step exclusion
process, introduced in \cite{guiol}. Let  $K=1$,  $k\in\N$, and $p(.)$ be a
jump kernel on $\Z$ satisfying assumptions \textit{(A1)--(A2)}. A
particle at $x$ performs a random walk with kernel $p(.)$ and jumps
to the first vacant site it finds along this walk, unless it returns
to $x$ or does not find an empty site within
$k$ steps, in which case it stays at $x$. \\ \\
To generalize this, let  $K\geq 1$,  $k\geq 1$, $c\in(0,1)$, and  $\mathcal D$
denote the set of functions $\beta=(\beta^1,\ldots,\beta^k)$ from
$\Z^k$ to  $(0,1]^k$  such that
\begin{eqnarray}\label{eq:(o)}
\beta^1(.)&\in&[c,1]\\
\label{eq:(i)}
\beta^i(.)&\geq&\beta^{i+1}(.),\,\forall i\in\{1,\ldots,k-1\}
\end{eqnarray}
In the sequel, an element of $\Z^k$ is denoted by
$\underline{z}=(z_1,\ldots,z_k)$.
Let $q$ be a probability distribution on $\Z^k$, and $\beta\in\mathcal D$.
We define the $(q,\beta)$-$k$ step $K$-exclusion process as follows.
A particle at $x$  (if some)  picks a $q$-distributed random vector $\underline{Z}=(Z_1,\ldots,Z_k)$,
and jumps to the  first  site $x+Z_i$ ($i\in\{1,\ldots,k\})$  with strictly less than $K$ particles
along the path $(x+Z_1,\ldots,x+Z_k)$, if such a site exists, with rate $\beta^i(\underline{Z})$.
Otherwise, it stays at $x$.
The $k$-step exclusion process corresponds to the particular case where  $K=1$,
$q$ is the distribution (hereafter denoted by $q^k_{RW}(p))$ of the first
$k$ steps  of a random walk with kernel $p(.)$ absorbed at $0$, and $\beta^i(\underline{z})=1$.
Outside the fact that $K$ can take values $\geq 1$, our model extends  $k$-step exclusion in different directions:  \\ 
\textit{(1)} The random path followed by the particle need not be a  Markov process.\\
\textit{(2)} The distribution $q$ is not necessarily supported on paths absorbed at 0.\\
\textit{(3)}  Different rates can be assigned to jumps according to the
number of steps, and the collection of these rates may depend on the path realization. \\ \\
Next, disorder is introduced:
the environment is a field
$\alpha=((q_x,\beta_x):\,x\in\Z)\in{\mathbf A}:=(\mathcal
P(\Z^k)\times\mathcal D)^\Z$. For a given realization of the
environment,
the distribution of the path $\underline{Z}$ picked by a particle
at $x$ is $q_x$, and the rate at which it jumps to $x+Z_i$ is $\beta^i_x(\underline{Z})$.
The corresponding generator is given by  \eqref{generic_form} with $c_\alpha=\sum_{i=1}^k c_\alpha^i$, where
(with the convention that an empty product is equal to $1$)
$$c_\alpha^i(x,y,\eta )={\bf 1}_{\{\eta(x)>0\}}{\bf 1}_{\{\eta(y)<K\}} \int\left[
\beta^i_x(\underline{z}) {\bf 1}_{\{x+z_i=y\}} \prod_{j=1}^{i-1}{\bf 1}_{\{\eta(x+z_j)=K\}}
\right]\,dq_x(\underline{z})
$$
 The distribution $Q$ of the environment on $\mathbf A$  is assumed ergodic
with respect to the space shift  $\tau_y$, where
$\tau_y\alpha=((q_{x+y},\beta_{x+y}):\,x\in\Z)$. \\ \\
For the existence of the process and graphical construction below,
and for Proposition \ref{prop_irred},  sufficient assumptions to
replace \textit{(A1)--(A2)} are
\begin{eqnarray}\label{irreducibility_kstep}
\inf_{x\in\Z} q^1_x(.) & \geq & c p(.)\\\label{summability_kstep}
\sup_{i=1,...,k}\sup_{x\in\Z}q^i_x(.) & \leq &  c^{-1}P(.)
\end{eqnarray}
 for some constant $c>0$, where  $q_x^i$ denotes the $i$-th marginal of $q_x$,
and $p(.)$, resp. $P(.)$, are probability distributions satisfying \textit{(A1)}, resp. \textit{(A2)}.
 To write the microscopic flux and define a graphical construction, we introduce the following notation: for
$(x,\underline{z},\eta)\in\Z\times\Z^k\times{\mathbf X}$,  $\beta\in\mathcal D$ and $u\in[0,1]$, 
\begin{eqnarray*}\label{nbsteps} N(x,\underline{z},\eta) &=&
\inf\left\{i\in\{1,\ldots,k\}:\,\eta\left(x+z_i\right)<K\right\}
\mbox{ with} \inf\emptyset=+\infty \\ 
\label{finloc} Y(x,\underline{z},\eta) &=& \left\{
\ba{lll}
x+z_{N(x,\underline{z},\eta)} & \mbox{if} & N(x,\underline{z},\eta)<+\infty\\
x & \mbox{if} & N(x,\underline{z},\eta)=+\infty
\ea
\right. \\
{{\mathcal T}_0}^{x,\underline{z},\beta,u}\eta&=&\left\{
\ba{lll}
\eta^{x,Y(x,\underline{z},\eta)} & \mbox{if} & \eta(x)>0\mbox{ and }u< \beta^{N(x,\underline{z},\eta)}(\underline{z}) \\
\eta &  & \mbox{otherwise} 
\ea
\right.
\end{eqnarray*} 
 (where the definition of $\beta^{+\infty}(\underline{z})$ has no importance).
 With these notations, we have
\begin{eqnarray}
c_\alpha(x,y,\eta) & = & {\bf 1}_{\{\eta(x)>0\}}\Exp_{q_0}\left[
\beta^{N(x,\underline{Z},\eta)}_0 {\bf 1}_{\{Y(x,\underline{Z},\eta)=y\}}
\right]\label{rate_genkstep} \\
\label{flux_genkstep}
\widetilde{\jmath}(\alpha,\eta) & = & {\bf 1}_{\{\eta(0)>0\}}\Exp_{q_0}\left[
\beta^{N(0,\underline{Z},\eta)}_0 Y(0,\underline{Z},\eta)
\right]
\end{eqnarray}
where expectation is with respect to $\underline{Z}$. Since 
$$\left|
\beta^{N(0,\underline{Z},\eta)}_0 Y(0,\underline{Z},\eta)-
\beta^{N(0,\underline{Z},\xi)}_0 Y(0,\underline{Z},\xi)
\right|\leq
2\sum_{i=1}^k |Z_i|\sum_{i=1}^k|\eta(Z_i)-\xi(Z_i)|
$$
\eqref{standard_coupling}--\eqref{compare_fluxes} yield for $G^Q$ the Lipschitz constant 
$V=2k^2c^{-1}\sum_{z\in\Z}|z|P(z)$.\\ \\
Let $\mathcal V=[0,1]\times[0,1]$,
$m=\lambda_{[0,1]}\otimes\lambda_{[0,1]}$. For each  probability
distribution $q$ on $\Z^k$, there exists a mapping
$F_q:[0,1]\to\Z^k$ such that $F_q(V_1)$ has distribution $q$ if
$V_1$ is uniformly distributed on $[0,1]$. Then the transformation
$\mathcal T$ in \eqref{update_rule} is defined by  (with
$v=(v_1,v_2)$ and  $\alpha=((q_x,\beta_x):\,x\in\Z))$
\be\label{update_kstep}
{\mathcal T}^{\alpha,x,v}\eta={\mathcal T}_0^{x,F_{q_x}(v_1),\beta_x(F_{q_x}(v_1)),v_2}\eta
\ee
Strong attractiveness of our process will follow from
\begin{lemma}\label{attractive_kstep}
For every $(x,\underline{z},u)\in\Z\times\Z^k\times[0,1]$,
${\mathcal T}_0^{x,\underline{z},\beta,u}$ is an increasing
mapping from ${\mathbf X}$ to ${\mathbf X}$.
\end{lemma}
\begin{proof}{lemma}{attractive_kstep}
Let $(\eta,\xi)\in{\mathbf X}^2$  with $\eta\leq\xi$. To prove that
${\mathcal T}_0^{x,\underline{z},\beta,u}\eta\leq {\mathcal
T}_0^{x,\underline{z},\beta,u}\xi$, since $\eta$ and $\xi$ can only
possibly change at sites $x$, $y:=Y(x,\underline{z},\eta)$  and
$y':=Y(x,\underline{z},\xi)$,  it is sufficient to verify the
inequality at these sites.
\\ \\
If $\xi(x)=0$, then by \eqref{update_kstep}, $\eta$ and $\xi$ are
both unchanged by ${\mathcal T}_0^{x,\underline{z},\beta,u}$. If
$\eta(x)=0<\xi(x)$, then  ${\mathcal
T}_0^{x,\underline{z},\beta,u}\xi(y')
\geq\xi(y')\geq\eta(y')={\mathcal T}_0^{x,\underline{z},\beta,u}\eta(y')$.  \\ \\
Now assume  $\eta(x)>0$. Then $\eta\leq\xi$ implies $N(x,\underline{z},\eta)\leq N(x,\underline{z},\xi)$. If $N(x,\underline{z},\eta)=+\infty$,  $\eta$ and $\xi$ are unchanged. If $N(x,\underline{z},\eta)<N(x,\underline{z},\xi)=+\infty$, then ${\mathcal T}_0^{x,\underline{z},\beta,u}\eta=\eta^{x,y}$ and  $\xi(y)=K$.  Thus,
${\mathcal T}_0^{x,\underline{z},\beta,u}\eta(x)=\eta(x)-1\leq\xi(x)={\mathcal T}_0^{x,\underline{z},\beta,u}\xi(x)$ and  ${\mathcal T}_0^{x,\underline{z},\beta,u}\xi(y)=\xi(y)=K\geq{\mathcal T}_0^{x,\underline{z},\beta,u}\eta(y)$.
If $N(x,\underline{z},\eta)=N(x,\underline{z},\xi)<+\infty$, then
$\beta^{N(x,\underline{z},\eta)}=\beta^{N(x,\underline{z},\xi)}=:\beta$.
If $u\geq\beta$ both $\eta$ and $\xi$ are unchanged. Otherwise
${\mathcal T}_0^{x,\underline{z},\beta,u}\eta=\eta^{x,y}$ and
${\mathcal T}_0^{x,\underline{z},\beta,u}\xi=\xi^{x,y}$, whence the
conclusion. Finally, assume
$N(x,\underline{z},\eta)<N(x,\underline{z},\xi)<+\infty$, hence
$\beta:=\beta^{N(x,\underline{z},\eta)}\geq\beta^{N(x,\underline{z},\xi)}=:\beta'$
by \eqref{eq:(i)} and  $\eta(y)<\xi(y)=K$.  If $u\geq\beta$, $\eta$ and $\xi$ are unchanged.
 If $u<\beta'$, then  ${\mathcal
T}_0^{x,\underline{z},\beta,u}\eta(y)=\eta(y)+1\leq\xi(y)={\mathcal
T}_0^{x,\underline{z},\beta,u}\xi(y)=K$  and ${\mathcal
T}_0^{x,\underline{z},\beta,u}\xi(y')=\xi(y')+1\geq {\mathcal
T}_0^{x,\underline{z},\beta,u}\eta(y')$.  If $\beta'\leq u<\beta$,
then  ${\mathcal T}_0^{x,\underline{z},\beta,u}\eta(x)=\eta(x)-1\leq
{\mathcal T}_0^{x,\underline{z},\beta,u}\xi(x)$ and   ${\mathcal
T}_0^{x,\underline{z},\beta,u}\eta(y)=\eta(y)+1\leq {\mathcal
T}_0^{x,\underline{z},\beta,u}\xi(y)=\xi(y)=K$.
\end{proof}
\mbox{}\\ \\
We now describe a few examples.\\ \\
 {\em Example 2.1.} Let  $K=1$,  $(\alpha_x:\,x\in\Z)$ be an ergodic $[c,1/c]$-valued
random field, and $r(.)$ be a
probability measure on $\Z$ satisfying  \textit{(A1)--(A2)}.  A disordered
version of the $k$-step exclusion process with jump kernel $r$ is obtained
by multiplying the rate of any jump starting from $x$ by $\alpha_x$. This
means that the random field $(q_x,\beta_x)_{x\in\Z}$ is defined by
$q_x=q^k_{RW}(r)$, and  $\beta_x(\underline{z})=(\alpha_x,\ldots,\alpha_x)$ for every $\underline{z}\in\Z^k$.\\ \\
 {\em Example 2.2.} Let $(\gamma_x,\iota_x)_{x\in\Z}$ be an ergodic
$[c,1]^{2k}$-valued random field, 
where $\gamma_x=(\gamma_x^n,\,1\le n\le k)$ and $\iota_x=(\iota_x^n,\,1\le n\le k)$. 
The random field $(q_x,\beta_x)_{x\in\Z}$ is defined by 
\begin{eqnarray*}
q_x&=&\frac{1}{2}\delta_{(1,2,\ldots,k)}+\frac{1}{2}\delta_{(-1,-2,\ldots,-k)}\\ \beta^i_x(1,2,\ldots,k)=2\gamma^i_x,&&\beta^i_x(-1,-2,\ldots,-k)=2\iota^i_x
\end{eqnarray*}
 Hence the rates are disordered but not the distribution of the random path
followed by particles: the stationary random field $(q_x)_{x\in\Z}$ is
deterministic and uniform. Here, the  jump rate and microscopic flux \eqref{rate_genkstep}--\eqref{flux_genkstep}  have a fairly explicit form:
\begin{eqnarray}
\label{fairly_explicit_rate_1} c_\alpha(x,y,\eta) & = & \gamma^{y-x}_x {\bf 1}_{\{\eta(x)>0\}}{\bf 1}_{\{\eta(y)<K\}}\prod_{z=x+1}^{y-1}{\bf 1}_{\{\eta(z)=K\}}\mbox{ if }y>x\\
\label{fairly_explicit_rate_2} c_\alpha(x,y,\eta) & = & \iota^{x-y}_x {\bf 1}_{\{\eta(x)>0\}}{\bf 1}_{\{\eta(y)<K\}}\prod_{z=y+1}^{x-1}{\bf 1}_{\{\eta(z)=K\}}\mbox{ if }y<x\\
\nonumber
\widetilde{\jmath}(\alpha,\eta) & = & \eta(0)\sum_{n=1}^k n\gamma^n_0(1-\eta(n))\prod_{j=1}^{n-1}\eta(j)\\
& - & \eta(0)\sum_{n=1}^k n\iota^n_0(1-\eta(-n))\prod_{j=1}^{n-1}\eta(-j)\label{fairly_explicit}
\end{eqnarray}
{\em Example 2.3.} Set $q_x=q^k_{RW}(r_x)$, for $(r_x)_{x\in\Z}$ an
ergodic random field with values in the probability measures on $\Z$
satisfying \textit{(A1)--(A2)}. The simplest case is
nearest-neighbor jumps, that is,
$r_x=p_x\delta_1+(1-p_x)\delta_{-1}$, where, for some
$c\in(0,1)$, $(p_x)_{x\in\Z}$ is an ergodic $[c,1/c]$-random
field.  
 Due to the nearest-neighbor assumption, a particle starting from $x$ can only jump to $y>x$ (resp. $y<x$) if $y$ is not full and all sites between $x$ and $y$
(resp. $y$ and $x$) are full. Hence,
the jump rate \eqref{rate_genkstep} is identical (see example below) to the one obtained by taking in  \eqref{fairly_explicit_rate_1}--\eqref{fairly_explicit_rate_2}
$$\gamma^n_x=\sum_{l=0}^{\lfloor(k-n)/2\rfloor}p_x^{n+l}(1-p_x)^l C_n(n+l,l),
\quad\iota^n_x=\sum_{l=0}^{\lfloor(k-n)/2\rfloor}(1-p_x)^{n+l}p_x^l C_n(n+l,l)
$$
for $n\in\{1,\ldots,k\}$, 
where $C_n(i,j)$, for $i,j\in\Z^+$ and $i+j>0$, is the number of paths  $(z_0=0,\ldots,z_{i+j})$ such that $0<z_m<n$
for  $m=1,\ldots,i+j-1$,  $|z_{m+1}-z_{m}|=1$ for $m=1,\ldots,i+j$, 
and ${\rm Card}\{m\in\{1,\ldots,i+j\}:\,z_m-z_{m-1}=1\}=i$.
With this choice of $\gamma^n_x$ and $\iota^n_x$, the microscopic flux is given by \eqref{fairly_explicit}.
 For instance if $k=5$, we obtain, for $n\in\{1,\ldots,k\}$:
\begin{eqnarray}
c_\alpha(x,x+n,\eta) & = & p_x^n{\bf 1}_{\{\eta(x)>0\}}{\bf 1}_{\{\eta(x+n)<K\}}\prod_{j=1}^{n-1}{\bf 1}_{\{\eta(x+j)=K\}}\quad\mbox{ if }n\neq 3\\
\nonumber
c_\alpha(x,x+3,\eta) & = & p_x^3[1+p_x(1-p_x)]\times\\
\label{step3_right}&&\qquad\eta(x)\eta(x+1)\eta(x+2)(1-\eta(x+3))\\
\nonumber
c_\alpha(x,x-n,\eta) & = & (1-p_x)^n{\bf 1}_{\{\eta(x)>0\}}{\bf 1}_{\{\eta(x-n)<K\}}\prod_{j=1}^{n-1}{\bf 1}_{\{\eta(x-j)=K\}}\quad\mbox{ if }n\neq 3\\
\nonumber
c_\alpha(x,x-3,\eta) & = & (1-p_x)^3[1+p_x(1-p_x)]\times\\
\nonumber&&\qquad\eta(x)\eta(x-1)\eta(x-2)(1-\eta(x-3))
\end{eqnarray}
Indeed,
for $n>0$ and $n\neq 3$, the only path from $x$ to $x+n$ that reaches $x+n$ in at most $k$ steps before returning to $0$ is $x,x+1,\ldots,x+n$.
For $n=3$, the  additional path $x,x+1,x+2,x+1,x+2,x+3$ yields the factor $p_x(1-p_x)$ in \eqref{step3_right}. For $n<0$, we change $p_x$ to $1-p_x$.\\ \\
Note that in this process a given particle does not follow a
random walk in random environment (RWRE) before it finds a  non full  site,
 but a homogeneous random walk depending (randomly) on its
initial location. For instance, in a  $3$-step process,  a particle
initially at $x\in\Z$ will follow the path
$x,x+1,x+2,x+1$ with probability $p_x^2(1-p_{x})$.\\ \\
{\em Example 2.4.} The same random field $(p_x)_{x\in\Z}$ gives a different
model if, at each transition, the selected particle follows a RWRE
$(X_n)_{n\ge 0}$ with transition probabilities
\be\label{rwre}\Prob(X_{n+1}=x+1|X_n=x)=p_x,\quad \Prob(X_{n+1}=x-1|X_n=x)=1-p_x\ee
That is, we let $q_x$ be the distribution of
$(X_1^x-x,\ldots,X_k^x-x)$, for $(X_n^x,\,1\le n\le k)$ a length $k$
Markov chain starting at $x$ with transition probabilities
\eqref{rwre}. There,   unlike in  Example 2.3 above, a particle
initially at $x\in\Z$ follows  the path $x,x+1,x+2,x+1$ with
probability $p_x p_{x+1}(1-p_{x+2})$.  The generator of this process is also identical 
to that of example 2.2, with $\gamma^n_x$ and $\iota^n_x$ of the form
$$
\gamma^n_x=\gamma^n(p_y:\,x\leq y< x+n),
\quad\iota^n_x=\iota^n(p_y:x-n<y\leq x)
$$
for some polynomial functions $\gamma^n,\iota^n:[0,1]^n\to[0,+\infty)$, where $n\in\{1,\ldots,k\}$.
\subsection{ $K$-exclusion process with speed change and traffic flow model}
Let $\mathcal K:=\{-k,\ldots,k\}\setminus\{0\}$, and 
$\alpha=((\upsilon(x),\beta^1_x):x\in\Z)$
be an ergodic $[0,+\infty)^{2k}\times(0,+\infty)$-valued  
field, where $\upsilon(x)=(\upsilon_z(x):\,
z\in{\mathcal K})$.
We define the following dynamics.  Set 
\begin{eqnarray*}
\Theta(x,\eta) & := & \{y\in\Z:\,y-x\in\mathcal K,\,\eta(y)<K\}\\
Z(\alpha,x,\eta) & := & \sum_{z\in\Theta(x,\eta)}\upsilon_{z-x}(x)
\end{eqnarray*}
In configuration $\eta$, if $Z(\alpha,x,\eta)>0$, a particle at $x$
picks a site $y$ at random in $\Theta(x,\eta)$ with probability
 $Z(\alpha,x,\eta)^{-1}\upsilon_{y-x}(x)$,  and  jumps to
this site at rate  $\beta^1_x$.  If  $Z(\alpha,x,\eta)=0$,  nothing
happens. For instance, if  $\upsilon_z(x)\equiv 1$, 
the particle uniformly chooses a site with strictly less than $K$
particles.  The corresponding
generator is  given by \eqref{generic_form}, with 
\[
c_\alpha(x,y,\eta)={\bf 1}_{\{\eta(x)>0\}}{\bf 1}_{\{Z(\alpha,x,\eta)>0\}}{\bf 1}_{\Theta(x,\eta)}(y)
Z(\alpha,x,\eta)^{-1}\upsilon_{y-x}(x)
\]
 Hence, the microscopic flux \eqref{other_flux} writes 
$$
\widetilde{\jmath}(\alpha,\eta)=\beta^1_0{\bf 1}_{\{\eta(0)>0\}}Z(\alpha,0,\eta)^{-1}\sum_{z\in\mathcal K}z\upsilon_z(0){\bf 1}_{\{\eta(z)<K\}}
$$
This process can be compared with a bond-disordered  
$K$- exclusion process in which a particle at $x$ jumps  to $y$ with rate
$\alpha(x,y)=\upsilon_{y-x}(x)$. The difference is that in the
latter, the particle could pick a location occupied by $K$ particles, in which case
the jump is suppressed. In the former, the particle first eliminates
sites occupied by $K$ particles
and picks a site occupied by strictly less than $K$ particles whenever there is at least one.
This results in a speed change  $K$-exclusion process,  that is  the jump rate from
$x$ to $y$ has the form  $c_{x,y}(\eta){\bf 1}_{\{\eta(x)>0\}}{\bf 1}_{\{\eta(y)<K\}}$.
To illustrate this, consider a nearest-neighbor example:  we take $K=1$,  $k=1$,
$\upsilon_1(x)=p(x)\in[0,1]$, $\upsilon_{-1}(x)=1-p(x)$. If sites
 $x-1$  and $x+1$ are free, in both processes the particle
at $x$ moves with rate $\beta^1_x$ to a site picked in $\{x-1,x+1\}$
with probabilities $p(x)$ and
$1-p(x)$. Now assume $x+1$ is free and $x-1$ occupied. If $p(x)=0$,
nothing happens in either process. If $p(x)>0$, at rate $\beta^1_x$,
the particle at $x$   moves to $x+1$ in the speed change process,
while in the bond-disordered process it  moves to $x+1$ with probability
$p(x)$ and attempts in vain to jump to $x-1$ with probability $1-p(x)$.\\ \\
Assume $K=1$, and consider the totally asymmetric case, where $\upsilon_z(x)=0$ for $z<0$.
Recalling that the totally asymmetric exclusion process is a classical
simplified model of single-lane traffic flow (without overtaking) where particles represent cars,
the above model can be viewed as a traffic-flow model with maximum overtaking distance $k$.
This is true also for  Example 2.2  in Subsection \ref{subsec:k-step},
in the totally asymmetric setting  $\iota_x^i=0,1\le i\le k$.  However in the
latter model, an overtaking car has  only one  choice for its new
position.
\\ \\
Though it is not clear from this formulation, we can rephrase this dynamics as a 
 $2k$-step model, which is thus  strongly attractive by Lemma \ref{attractive_kstep}.
To this end we take a random field of the form
$\beta_x=(\beta^1_x,\ldots,\beta^1_x)$, and define
$q_x:=q(\upsilon(x))$, where $q(\upsilon_z:\,z\in{\mathcal K})$ is
the distribution of a random self-avoiding path
$(Z_1,\ldots,Z_{2k})$ in $\mathcal K$ such that
\begin{eqnarray}\label{self_avoiding}
\mathbf{P}(Z_1=y)&=&\frac{\upsilon_y}{\sum_{z\in\mathcal K}
\upsilon_z}\\\label{self_avoiding2} {\mathbf
P}\left(Z_i=y|Z_1,\ldots,Z_{i-1}\right)&=&\frac{\upsilon_{y}}
{\sum_{z\in\mathcal K\backslash\{Z_1,\ldots,Z_{i-1}\}}
\upsilon_{z}}\quad\mbox{ for  } 2\le i\le 2k
\end{eqnarray}
For this model,  assumption \eqref{summability_kstep} is always satisfied,
while \eqref{irreducibility_kstep} reduces to
the existence of a constant $c>0$ and a probability distribution
$p(.)$ on $\Z$ satisfying assumption \textit{(A1)}, such that
$$
\inf_{x\in\Z}\upsilon_.(x)\geq c\,p(.)$$
 The link between the two models comes from
\begin{lemma}
\label{small_computation} Assume $(Z_1,\ldots,Z_{2k})\sim
q(\upsilon_z:\,z\in{\mathcal K})$. Let $\Theta$ be a nonempty subset
of $\{z\in{\mathcal K}:\upsilon_z\not=0\}$,
$\tau:=\inf\{i\in\{1,\ldots,2k\}:\,Z_i\in\Theta\}$,    and
$Y=Z_\tau$. Then
$$
{\mathbf P}\left(Y=y\right)={\bf 1}_{\Theta}(y)\frac{\upsilon_{y}}{
\sum_{y'\in\Theta}\upsilon_{y'}}
$$
\end{lemma}
\begin{proof}{lemma}{small_computation}
For all $t\geq 2$, let $\Theta_{t-1}$ be the set of  self-avoiding
paths $(z_1,\ldots,z_{t-1})$ of size $t-1$ on  ${\mathcal
K}\setminus\Theta$.  For $y\in\Theta$, by
\eqref{self_avoiding}--\eqref{self_avoiding2},
\begin{eqnarray*}
{\mathbf P}(Y=y) & = & \sum_{t=1}^{2k}{\mathbf P}(Z_t=y,\tau=t)\\
 &=&  \mathbf{P}(Z_1=y)+\\
 \sum_{t=2}^{2k}\sum_{(z_1,\ldots,z_{t-1})\in \Theta_{t-1}}
&&{\mathbf P}(Z_1=z_1,\ldots,Z_{t-1}=z_{t-1})
\frac{\upsilon_y}{
\sum_{z\in{\mathcal K}\backslash\{z_1,\ldots,z_{t-1}\}}\upsilon_z
}\\
 & = & C\upsilon_y
\end{eqnarray*}
where $C$ is independent of $y\in\Theta$, whence the result.
\end{proof}
\mbox{}\\ \\
\noindent {\bf Acknowledgments:}
K.R. was supported by NSF grant DMS 0104278.
We thank BCM at  TIMC - IMAG,
Universit\'{e}s de Rouen, Paris Descartes and Clermont 2,
and SUNY College at New Paltz, for hospitality.
\end{document}